\documentclass{imsart}
\usepackage{comment}
\usepackage{mathrsfs}
\RequirePackage[OT1]{fontenc}
\RequirePackage{amsthm,amsmath,amsfonts,color,mathrsfs,graphicx}
\RequirePackage{natbib}
\RequirePackage[colorlinks,citecolor=blue,urlcolor=blue]{hyperref}

\startlocaldefs
\numberwithin{equation}{section}
\theoremstyle{plain}
\endlocaldefs
\theoremstyle{definition}

\newcommand{\etal}{{\em et al.}}

\usepackage{float,paralist,enumerate}
\usepackage{enumitem, hyperref}

\makeatletter
\def\namedlabel#1#2{\begingroup
    #2%
    \def\@currentlabel{#2}%
    \phantomsection\label{#1}\endgroup
}
\makeatother

\newtheorem{thm}{Theorem}[section]
 \newtheorem{cor}{Corollary}[section]
 \newtheorem{lem}{Lemma}[section]
 \newtheorem{prop}{Proposition}[section]
 \theoremstyle{definition}
 \newtheorem{defn}{Definition}[section]
 \newtheorem{rmk}{Remark}[section]

\DeclareMathOperator{\var}{Var}   \DeclareMathOperator{\cov}{Cov}
\DeclareMathOperator{\tr}{tr}

\renewcommand{\(}{\left(}
\renewcommand{\)}{\right)}
\newcommand{\lj}{\left|}
\newcommand{\rj}{\right|}
\newcommand{\De}{\Delta}
\newcommand{\la}{\lambda}

\newcommand{\de}{\delta}
\newcommand{\al}{\alpha}
\renewcommand{\th}{\theta}
\newcommand{\si}{\sigma}

\newcommand{\ga}{\gamma}
\newcommand{\pmu}{\pmb\mu}
\newcommand{\pTh}{\pmb\Theta}
\newcommand{\pLa}{\pmb\Lambda}
\newcommand{\ep}{\epsilon}
\newcommand{\pSi}{\pmb\Sigma}

\newcommand{\pvep}{\pmb\varepsilon}
\newcommand{\vep}{\varepsilon}
\newcommand{\pep}{\pmb\epsilon}

\def\ICV{\mbox{ICV}}

\def\RCV{\mbox{RCV}}
\def\PAV{\mbox{PAV}}
\def\$\cB_m${\mbox{$\cB_m$}}
\DeclareMathOperator*{\argmin}{arg\,min}

\newcommand{\A}{{\bf A}}
\newcommand{\B}{{\bf B}}

\newcommand{\I}{{\bf I}}
\newcommand{\e}{{\bf e}}
\newcommand{\R}{{\bf R}}
\newcommand{\X}{{\bf X}}
\newcommand{\Y}{{\bf Y}}
\newcommand{\Z}{{\bf Z}}
\newcommand{\W}{{\bf W}}
\newcommand{\V}{{\bf V}}
\renewcommand{\S}{{\bf S}}
\newcommand{\x}{{\bf x}}

\newcommand{\toop}{\stackrel{p}{\longrightarrow}}

\newcommand{\cD}{{\mathcal D}}
\newcommand{\cF}{{\mathcal F}}
\newcommand{\cL}{{\mathcal L}}
\newcommand{\cA}{{\mathcal A}}
\newcommand{\cB}{{\mathcal B}}
\newcommand{\bC}{{\mathbb C}}
\newcommand{\bR}{{\mathbb R}}
\newcommand{\sD}{{\mathscr{D}}}

\newcommand{\zz}[1]{\mathbb{#1}}

\def\tSi {{\widetilde{\Sigma}}}
\def\tpSi {{\widetilde{\pmb\Sigma}}}

\def\bpSi {{\breve{\pmb\Sigma}}}
\newcommand{\ol}{\overline}
\newcommand{\ul}{\underline}
\newcommand{\wt}{\widetilde}
\newcommand{\wh}{\widehat}
\def\eps{\varepsilon}

\def\q{\quad}

\newcommand{\ip}[1]{\langle #1 \rangle}

\definecolor{darkviolet}{rgb}{0.58, 0.0, 0.83}

\renewcommand{\theenumi}{\Alph{enumi}}

 \makeatletter
 \renewcommand{\p@enumii}{\theenumi.}
 \makeatother

\begin{document}
\begin{frontmatter}

\title{On the inference about the spectra of high-dimensional covariance matrix based on noisy observations}
\begin{Large}
--with applications to integrated covolatility matrix inference in the presence of microstructure noise
\end{Large}
\runtitle{Inferring spectra of HD Cov based on Noisy Observations}

\begin{aug}
\author{\fnms{Ningning} \snm{Xia}\ead[label=e1]{xia.ningning@mail.shufe.edu.cn}\thanksref{t1}}
\and
\author{\fnms{Xinghua} \snm{Zheng}\ead[label=e2]{xhzheng@ust.hk}\thanksref{t2}}

\thankstext{t1}{Research partially supported by NSFC 11501348 and  Shanghai Pujiang Program 15PJ1402300.}
\thankstext{t2}{Research partially supported by DAG (HKUST) and GRF 606811 of the HKSAR.}
\runauthor{N. Xia and X. Zheng}

\affiliation{Shanghai University of Finance and Economics, and Hong Kong University of Science and Technology}

\address{School of Statistics and Management, Shanghai\\
Key Laboratory of Financial Information Technology,\\
Shanghai University of Finance and Economics\\
777 Guo Ding Road, China, 200433\\
\printead{e1}}

\address{Department of information systems,\\
business statistics and operations management\\
Hong Kong University of Science and Technology\\
Clear Water Bay, Kowloon, Hong Kong\\
\printead{e2}}
\end{aug}

\begin{abstract}
In practice,  observations are often contaminated by noise, making the resulting sample covariance matrix to be an information-plus-noise-type covariance matrix.
Aiming to make inferences about the spectra of the underlying true covariance matrix under such a situation,  we establish an asymptotic relationship that describes how the limiting spectral distribution of (true) sample covariance matrices depends on that of information-plus-noise-type sample covariance matrices.
As an application, we consider the inference about the  spectra of integrated covolatility (ICV) matrices of high-dimensional diffusion processes based on high-frequency data with microstructure noise.  The (slightly modified) pre-averaging estimator is an
information-plus-noise-type covariance matrix, and the aforementioned result, together with a (generalized) connection between the spectral distribution of true sample covariance matrices and that of the population covariance matrix, enables us to propose a two-step procedure to estimate the spectral distribution of ICV for a class of diffusion processes. An alternative estimator is further proposed, which possesses two desirable properties: it eliminates the impact of microstructure noise, and its limiting spectral distribution depends only on that of the ICV through the standard Mar\v{c}enko-Pastur equation.
Numerical studies demonstrate that our proposed methods can be used to estimate the spectra of the underlying covariance matrix based on noisy observations.
\end{abstract}



\end{frontmatter}

\section{Introduction and Main Results}\label{sec:intro}

\subsection{Motivation}\label{ssec:motivation}
Covariance structure is of fundamental importance in multivariate analysis and all kinds of applications. While in classical low-dimensional setting, a usually unknown covariance structure can be estimated by the sample covariance matrix, in the high-dimensional setting, it is now well understood that the sample covariance matrix is not a consistent estimator. What is even worse, in many applications the observations are contaminated, and below we explain one such setting that motivates this work. Similar situations certainly arise in many other settings.

Our motivating question arises in the context of estimating the so-called integrated covariance matrix of high-dimensional diffusion process, with applications towards the study of stock price processes. More specifically, suppose that we have $p$ stocks whose (efficient) log price processes are denoted by $(X_t^{j})$ for $j=1,\ldots,p$. Let $\X_t=(X_t^{1},\ldots,X_t^{p})^T$. A widely used model for $(\X_t)$ is
\begin{eqnarray}
d\X_t=\pmu_t\, dt + \pTh_t\, d\W_t, ~~~~~ t\in [0,1],
\label{diffusion}
\end{eqnarray}
where $\pmu_t=(\mu_t^{1},\ldots,\mu_t^{p})^T$ is a $p$-dimensional drift process,  $\pTh_t$ is a $p\times p$ matrix for any $t$, and is called the covolatility process, and $(\W_t)$ is a $p$-dimensional standard Brownian motion. The interval $[0,1]$ stands for the time period of interest, say, one day. The integrated covariance (ICV) matrix refers to
\[
\ICV:= \ \int_0^1 {\pmb\Theta}_t{\pmb\Theta}_t^T \, dt.
\]
The ICV matrix, in particular, its spectrum (i.e., its set of eigenvalues) plays an important role in financial applications such as factor analysis and risk management.

A classical estimator of the ICV matrix is the so-called realized covariance (RCV) matrix, which relies on the assumption that one could observe $(\X_t)$ at high frequency. More specifically, suppose that $(\X_t)$ could be observed at time points $t_i=i/n$ for $i=0,1,\ldots,n$. Then, the RCV matrix is defined as
\begin{eqnarray}
\RCV=\ \sum_{i=1}^n \De\X_i\(\De\X_i\)^T,
\label{RCV}
\end{eqnarray}
where
\[
\De\X_i= \( \begin{array}{c}
\De X_i^{1}\\
\vdots\\
\De X_i^{p} \end{array} \)
:= \( \begin{array}{c}
X_{t_i}^{1}-X_{t_{i-1}}^{1}\\
\vdots\\
X_{t_i}^{p}-X_{t_{i-1}}^{p} \end{array} \)
\]
stands for the vector of log returns over the period $[(i-1)/n, i/n]$.
Consistency as well as central limit theorems for RCV matrix under such a setting \emph{and} when the dimension $p$ is fixed is well unknown,  see, for example, \cite{AB98},
\cite{ABDL01}, \cite{BNS02}, \cite{jacodprotter98}, \cite{myklandzhang06},
among others.

However, in practice, the observed prices are always contaminated by the market microstructure noise. The market microstructure noise is induced by various frictions in the trading process such as the bid-ask spread, asymmetric information of traders, the discreteness of price, etc. Despite the small size,
market microstructure noise accumulates at high frequency and affects badly
the inferences about the efficient price processes.  In practice, as is pointed out in \cite{LPS15} where a careful comparison between various volatility estimators and the 5-min realized volatility is carried out, when the sampling frequency is higher than one observation per five minutes, the microstructure noise is usually no longer negligible, and the following additive model has been widely adopted in recent studies on volatility estimation:
\begin{equation}\label{eq:noise_additive}
\Y_{t_i}=\X_{t_i}+\pvep_i, ~~~ i=1,\cdots,n,
\end{equation}
where $\Y_{t_i}=(Y_t^{1},\ldots,Y_t^{p})^T$ denote the observations, $\pvep_i=(\eps_i^{1},\ldots,\eps_i^{p})^T$ denote the noise, which are \hbox{i.i.d.},
independent of $(\X_t)$, with $E(\pvep_i)=0$ and certain covariance matrix $\pSi_e$.

Observe that under \eqref{eq:noise_additive}, the observed log-returns $\Delta \Y_{t_i}:=\Y_{t_i} - \Y_{t_{i-1}}$ relates to the true log-returns $\Delta \X_{t_i}$ by the following equation
\begin{equation}\label{eq:signal_noise_additive}
\Delta \Y_{t_i}=\Delta \X_{t_i}+\Delta\pvep_i,  ~~~ i=1,\cdots,n,
\end{equation}
where, as usual, $\Delta \pvep_{i}:=\pvep_{i} - \pvep_{{i-1}}$. We are therefore in a noisy observation setting where the observations are contaminated by additive noise. Such a setting forms the basis of the current work.

Besides microstructure noise, there is another issue which is due to asynchronous trading. In practice, different stocks are traded at different times, consequently, the tick-by-tick data are not observed synchronously. There are several existing methods for synchronizing data, like the refresh times (\cite{BHLS11}) and previous tick method (\cite{Zhang11}). Compared with microstructure noise, asynchronicity is less an issue. For example, as is pointed out in \cite{Zhang11}, asynchronicity does not induce bias in the two-scales estimator,  and even the
asymptotic variance is the same  as if there is no asynchronicity. While we do not seek a rigorous proof to avoid the paper being unnecessarily lengthy, we believe our methods  to be introduced below work just as well for asynchronous data. The reason, roughly speaking, is as follows. Take the previous tick method for example. Here we choose a (usually equally spaced) grid of time points $0=t_0<t_1<\ldots<t_n=1$, and for each stock $j$, for each time $t_i$, let $\tau_i^j$ be the latest transaction time before $t_i$. One then acts as if one observes $Y_{\tau_i}^j$ at time $t_i$ for stock $j$. With the original additive model at time $\tau_i^j$:
\[
Y_{\tau_i}^j = X_{\tau_i}^j + \eps_{i}^j,
\]
we have at time $t_i$,
\[
Y_{t_i}^j:=Y_{\tau_i}^j = X_{t_i}^j +\left(( X_{\tau_i}^j - X_{t_i}^j)  + \eps_{i}^j\right).
\]
In other words, the asynchronicity induces an additional error $( X_{\tau_i}^j - X_{t_i}^j)$. The error is however diminishingly small as the sampling frequency $n\to\infty$ since $X_{\tau_i}^j - X_{t_i}^j=O_p(\sqrt{t_i - \tau_i^j}) = o_p(1)$. To sum up, even though  asynchronicity  induces an additional error, the error is of negligible order compared with the microstructure noise $(\eps_i^j)$. Henceforth we shall stick to the model \eqref{eq:signal_noise_additive}.

One striking feature in \eqref{eq:signal_noise_additive} that differs from most noisy observation settings is that as the observation frequency $n$ goes to infinity, the signal, namely, the true log-return $\Delta \X_{t_i}$ becomes diminishingly small, while the noise $\Delta\pvep_i$ remains of the same order of magnitude, and therefore the signal-to-noise ratio goes to 0. In the next section we will explain a method that fixes this issue. Our first main result, Theorem \ref{thm:LSD_signal_noise}, applies to a general setting where the signal and noise are of a same order of magnitude.

\subsection{Pre-averaging approach}\label{ssec:preaveraging}
The pre-averaging (PAV) approach is introduced in \cite{JLMPV09}, \cite{PV09} and \cite{CKP10} to deal with microstructure noise. Other approaches include the  two-scales estimator (\cite{ZMA05},\cite{Zhang11}),  realized kernel (\cite{BHLS08},\cite{BHLS11}) and  quasi-maximum likelihood method (\cite{Xiu10},\cite{AFX10}). We shall use a slight variant of the PAV approach in this work. First, choose a window length $k$. Then, group the intervals $[(i-1)/n,{i}/{n}]$, $i=1,\ldots,2k\cdot\lfloor{n}/(2k)\rfloor$ into
$m:=\lfloor{n}/(2k)\rfloor$ pairs of non-overlapping windows, each of width $(2k)/n$, where $\lfloor\cdot\rfloor$ denotes rounding down to the nearest integer.
Introduce the following notation for any process ${\bf V}=({\bf V}_t)_{t\geq 0}$,
\[
\Delta {\bf V}_i={\bf V}_{i/n}-{\bf V}_{(i-1)/n},
~~\overline{{\bf V}}_i=\frac{1}{k}\sum_{j=0}^{k-1}{\bf V}_{((i-1)k+j)/n},
~~{\rm and}~~
\Delta\overline{{\bf V}}_{2i}=\overline{{\bf V}}_{2i}-\overline{{\bf V}}_{2i-1}.
\]
With such notation, the observed return based on the pre-averaged price becomes
\begin{equation}\label{eq:rtn_pav}
\Delta\ol{\Y}_{2i}=\Delta\ol{\X}_{2i} + \Delta\ol{\pvep}_{2i}.
\end{equation}
One key observation is that if $k$ is chosen to be of order $\sqrt{n}$ (which is the order chosen in \cite{JLMPV09}, \cite{PV09} and \cite{CKP10}), then in \eqref{eq:rtn_pav} the ``signal'' $\De\ol\X_{2i}$ and ``noise'' $\De\ol\pvep_{2i}$ can be shown to be of the same order of magnitude.
Our version of PAV matrix is then just the sample covariance matrix based on these returns:
\begin{equation}\label{PAV}
\aligned
\PAV&:=\sum_{i=1}^m \left(\Delta\overline{{\bf Y}}_{2i}\right)\left(\Delta\overline{{\bf Y}}_{2i}\right)^T.
\endaligned
\end{equation}

\subsection{From signal to signal-plus-noise and back}\label{ssec:existing_work}
We first recall some concepts in multivariate analysis.  For any $p\times p$ symmetric matrix $\pSi$, suppose that its eigenvalues are
$\lambda_1,\ldots,\lambda_p$, then its \emph{empirical
spectral distribution} (ESD) is defined as
\[
   F^{\pSi}(x):=\frac{1}{p} \# \{j: \lambda_j \leq
   x\},\quad \mbox{for } x\in\zz{R}.
\]
The limit of ESD as $p\to\infty$, if exists, is referred to as the \emph{limiting
spectral distribution}, or LSD for short.

The matrix $\PAV$ can be viewed as the sample covariance matrix based on observations $\De\ol{\X}_{2i}+\De\ol{\pvep}_{2i}$, which model the situation of information vector $\De\ol{\X}_{2i}$ being contaminated by  additive noise $\De\ol{\pvep}_{2i}$.
\cite{DS2007a} consider such information-plus-noise-type sample covariance matrices as
\[
\S_n=\dfrac{1}{n}\(\A_n+\si\pvep_n\)\(\A_n+\si\pvep_n\)^T,
\]
where  $\pvep_n$ is independent of  $(\A_n)_{p\times n}$, and  consists of \hbox{i.i.d.} entries with zero mean and unit variance.
Write
$\cA_n:= \A_n\A_n^T/n$.
The authors show that if $F^{\cA_n}$ converges, then so does $F^{\S_n}$. They further show  how the LSD of~$\S_n$ depends on that of $\cA_n$ (see equation (1.1) therein).

In this article, we investigate the problem from a different angle. We shall show how the LSD of $\cA_n$ depends on that of $\S_n$.  Our motivation for seeking such a relationship is that, in practice,  we are usually  interested in making inferences about signals $\A_n$ based on noisy observations $\A_n+\si\pvep_n$. Therefore, a more practically relevant result is a relationship that describes how the LSD of $\cA_n$ depends on that of $\S_n$. Let us mention that inverting the aforementioned relationships is in general notoriously difficult. For example, the Mar\v{c}enko-Pastur equation, which is very similar to equation (1.1) in \cite{DS2007a} and describes how the LSD of the sample covariance matrix depends on that of the population covariance matrix, is long-established, but it was only a few years ago that researchers [\cite{ElKaroui08}; \cite{Mestre08};  \cite{BCY10} \hbox{etc.}] realized how the (unobservable) LSD of the population covariance matrix can be recovered based on the (observable) LSD of the sample covariance matrix. Our first result, Theorem \ref{thm:LSD_signal_noise} below,
provides an approach that allows one to derive the LSD of $\cA_n$ based on that of $\S_n$.

We first collect some notation that will be used throughout the article.  For any real  matrix ${\bf A}$, $\|{\bf A}\|=\sqrt{\lambda_{\textrm{max}}({\bf A}{\bf A}^T)}$ denotes its spectral norm, where ${\bf A}^T$ is the  transpose of  ${\bf A}$, and $\lambda_{\textrm{max}}$ denotes the largest eigenvalue.
For any nonnegative definite matrix $\B$, $\B^{1/2}$ denotes its square root matrix.
For any $z\in\mathbb{C}$, write $\Re(z)$ and $\Im(z)$ as its real and imaginary parts, respectively, and $\bar{z}$ as its complex conjugate. For any distribution $F$, $m_{F}(\cdot)$ denotes its Stieltjes transform defined as
\[
m_{F}(z)= \ \int \ \frac{1}{\lambda-z} \ dF(\lambda), ~~~ {\rm for}~ z\in\mathbb{C}^+{:=}\{z\in\mathbb{C}: \Im(z)>0\}.
\]
In particular, for any Hermitian matrix~$\pSi$ with eigenvalues $\lambda_1,\ldots,\lambda_p$, the Stieltjes transform of its ESD, denoted by $m_{\pSi}(\cdot):=m_{F^{\pSi}}(\cdot)$, is given by
\[
m_{\pSi}(z)= \frac{\tr((\pSi - z\I)^{-1})}{p}, ~~~ {\rm for}~ z\in\mathbb{C}^+,
\]
where $\I$ is the identity matrix.
For any vector $\x$, $|\x|$ stands for its Euclidean norm.
Finally, ``$\stackrel{d}{=}$'' stands for ``equal in distribution'', $\stackrel{\mathcal{D}}{\to}$ denotes weak convergence,
$\toop$ means convergence in probability, $Y_{n}=O_p (f (n))$ means that the sequence
$(|Y_{n}|/f (n))$ is tight, and $Y_{n}=o_p (f (n))$ means that $Y_{n}/f (n) \toop 0$.

We now present our first result about how the LSD of $\cA_n$ depends on that of $\S_n$.
\begin{thm}\label{thm:LSD_signal_noise}
Suppose that ${\bf S}_n=\dfrac{1}{n}({\bf A}_n+\sigma_n \pvep_n)({\bf A}_n+\sigma_n {\pmb \varepsilon}_n)^T$, where
\begin{compactenum}\setcounter{enumi}{1}
 \item[]
 \begin{compactenum}
  \item\label{asm:FA_conv} ${\bf A}_n$ is $p\times n$,  independent of $\pvep_n$, and with $\cA_n=(1/n)\A_n\A_n^T$, we have
      $F^{\cA_n}\stackrel{\mathcal{D}}{\to}F^{\cA}$, where $F^{\cA}$ is a  probability distribution with Stieltjes transform denoted by $m_{\cA}(\cdot)$;
  \item\label{asm:sigma_n_conv} $\sigma_n\geq 0$ with $\lim_{n\to\infty} \sigma_n=\sigma\in[0,\infty) $;
  \item\label{asm:eps} $\pvep_n=(\ep_{ij})$ is $p\times n$ with the entries $\ep_{ij}$ being \hbox{\hbox{i.i.d.}} and centered with unit variance; and
  \item\label{asm:yn_conv} $n=n(p)$ with $y_n=p/n\to y >0$ as $p\to\infty$.
\end{compactenum}
\end{compactenum}
Then, almost surely, the ESD of $\S_n$ converges in distribution to a probability distribution  $F$.
Moreover,  if $F$ is supported by a finite interval $[a,b]$ with $a>0$ and possibly has a point mass at 0,
then for all $z\in\bC^+$ such that the integral on the right hand side of \eqref{eqn:LSD_signal_to_noisy} below is well-defined,
$m_{\cA}(z)$ is determined by $F$ in that it uniquely solves the following equation
\begin{eqnarray}
~~~~~~~~
m_{\cA}(z) = \displaystyle\int
\dfrac{dF(\tau)}{\dfrac{\tau}{1-y\si^2m_{\cA}(z)}-z(1-y\si^2m_{\cA}(z))+\si^2(y-1)}
\label{eqn:LSD_signal_to_noisy}
\end{eqnarray}
in the set
\begin{equation}\label{dfn:mA_domain}
D_{\cA}:=
\{\xi\in\bC: ~
z(1-y\si^2 \xi)^2- \si^2(y-1)(1-y\si^2 \xi) \in\bC^+\}.
\end{equation}
\end{thm}

\begin{rmk}\label{rmk:alpha_range}
Since $m_{\cA}(z)\to 0$ and $z m_{\cA}(z)\to -1$ as $\Im(z)\to \infty$, the imaginary part of the denominator of the integrand on the right hand side of \eqref{eqn:LSD_signal_to_noisy} is asymptotically equivalent to $-\Im(z)$ as $\Im(z)\to \infty$, and so the integral is well-defined for all $z$ with $\Im(z)$ sufficiently large.
Note further that by the uniqueness theorem for analytic functions, knowing the values of $m_{\cA}(z)$ for $z$ with $\Im(z)$ sufficiently large is sufficient to determine $m_{\cA}(z)$ for all $z\in\bC^+.$
\end{rmk}

In practice, as the ESD of ${\bf S}_n$ is observable, we can replace $F$ with $F^{{\bf S}_n}$ and solve equation \eqref{eqn:LSD_signal_to_noisy} for $m_{\cA_n}(z)$. Since  $m_{\cA_n}(z)$ fully characterizes the ESD of $\cA_n$, this allows us to make inferences about the covariance structure of the underlying signals $A_n$.
In the simulation studies we explain in detail about how to implement this procedure in practice.

\subsection{Applications to PAV}\label{ssec:LSD_PAV}

The term $\Delta\overline{{\bf V}}_{2i}$ can be written in a more clear form by using the triangular kernel:
\begin{equation}\label{DelV}
\aligned
\Delta\overline{{\bf V}}_{2i}
=&\frac{1}{k}\sum_{j=0}^{k-1}\left({\bf V}_{((2i-1)k+j)/n}-{\bf V}_{((2i-2)k+j)/n}\right)\\
=&\frac{1}{k}\sum_{j=0}^{k-1} \sum_{\ell=1}^k \Delta {\bf V}_{(2i-2)k+j+\ell}\\
=&\sum_{|j|<k} \left(1-\frac{|j|}{k}\right) \Delta{\bf V}_{(2i-1)k+j}.
\endaligned
\end{equation}
Based on this, one can show that if dimension $p$ is fixed, then
\begin{equation}\label{eq:conv_noiseless_pav}
\cA_m:=\sum_{i=1}^m \Delta\overline{{\bf X}}_{2i} \cdot (\Delta\overline{{\bf X}}_{2i})^T \toop  \frac{\ICV}{3} \q  \mbox{as}\q n\to\infty.
\end{equation}
It is also easy to verify that
\[
\De\ol{\pvep}_{2i} \stackrel{d}{=} \sqrt{\dfrac 2k}  \ \e_i,
\]
where $\e_i$'s are \hbox{i.i.d.} random vectors with zero mean and covariance matrix $\pSi_e$.

The following Corollary is a direct consequence of Theorem \ref{thm:LSD_signal_noise}.
\begin{cor}\label{cor:A_PAV}
Suppose that
\begin{compactenum}\setcounter{enumi}{2}
 \item[]
 \begin{compactenum}
  \item\label{asm:X} for all $p$, $(\X_t)$ is a $p$-dimensional process for some drift process $(\pmu_t)$ and covolatility process $(\pTh_t)$ defined in \eqref{diffusion};
  \item\label{asm:A_conv} the ESD of $\cA_m$ converges to a probability distribution $F^{\cA}$ with Stieltjes transform denoted by $m_{\cA}(z)$;
  \item\label{asm:noise} the noise $(\pvep_i)_{1\le i\le n}$ are independent of $(\X_t)$, and are i.i.d. with zero mean and covariance matrix $\pSi_e=\si_p^2 \I$ for some $\si_p>0$ and $\si_p\to\si_e>0$ as $p\to\infty$;
  \item\label{asm:k_PAV} $k=\lfloor\theta\sqrt{n}\rfloor$ for some $\th\in(0,\infty)$, and $m=\lfloor\frac{n}{2k}\rfloor$ satisfy that $\lim_{p\to\infty} p/m=y>0$.
\end{compactenum}
\end{compactenum}
Then, almost surely, the ESD of PAV defined in \eqref{PAV} converges in distribution of a probability distribution $F$.
Moreover, if $F$ is supported by a finite interval $[a,b]$ with $a>0$ and possibly has a point mass at 0,
then for all $z\in\bC^+$ such that the integral on the right hand side of \eqref{eqn:m_A} below is well-defined,
$m_{\cA}(z)$ is determined by $F$ in that it uniquely solves the following equation
\begin{equation}\label{eqn:m_A}
m_{\cA}(z)= \displaystyle\int \dfrac{dF(\tau)}{\dfrac{\tau}{1-y\th^{-2}\si_e^2m_{\cA}(z)}-z(1-y\th^{-2}\si_e^2 m_{\cA}(z))+\th^{-2}\si_e^2(y-1)}
\end{equation}
in the set
\[
D_{\cA}':=\{\xi\in\bC: ~
z(1-y\th^{-2}\si_e^2 \xi)^2- \th^{-2}\si_e^2(y-1)(1-y\th^{-2}\si_e^2\xi) \in\bC^+\}.
\]
\end{cor}

\begin{rmk}\label{rmk:thm_apply_general_case}
Although Corollary \ref{cor:A_PAV} is stated for the case when noise components  have the same standard deviations, it can readily be applied to the case when the covariance matrix  ${\pmb\Sigma}_e$  is a general diagonal matrix, say ${\rm diag}(d_1^2,\ldots,d_p^2)$. To see this, let $d_{max}^2=\max(d_1^2,\ldots,d_p^2)$. We can then artificially add additional $\tilde{\pvep}_i$ to the original observations, where $\tilde{\pvep}_i$ are independent of $\pvep_i$, and are \hbox{i.i.d.} with zero mean  and covariance matrix $\widetilde{{\pmb\Sigma}}_e={\rm diag}(d_{max}^2-d_1^2,\ldots,d_{max}^2-d_p^2)$. The noise components in the modified observations then have the same standard deviation $d_{max}$, and Corollary \ref{cor:A_PAV} can  be applied. Note that the variances, $d_1^2,\ldots,d_p^2$, can be consistently estimated, see, for example, Theorem A.1 in \cite{ZMA05}. A similar remark applies to Theorem \ref{thm:LSD_signal_noise}.
\end{rmk}

\subsection{Further inference about ICV}
Corollary \ref{cor:A_PAV} allows us to estimate the ESD of $\cA_m$. In light of the convergence \eqref{eq:conv_noiseless_pav}, this would have been sufficient for us to make inferences about the ICV if  the convergence \eqref{eq:conv_noiseless_pav} held as well in the high-dimensional case. Unfortunately, it is not the case, and a further step to go from $\cA_m$ to \ICV\ is needed. Such an inference is generally impossible, as can be seen as follows. \ICV\ is an integral $\int_0^1 {\pmb\Theta}_t{\pmb\Theta}_t^T \, dt$. In the simple situation where $\pmu_t\equiv 0$ and $\pTh_t$ is deterministic, the building blocks in defining $\cA_m$, $\Delta{\bf X}_{i}$, are multivariate normals with mean 0 and covariance matrices $\int_{(i-1)/n}^{i/n}{\pmb\Theta}_t{\pmb\Theta}_t^T\, dt.$ The bottom line is, all the $n$ covariance matrices, $\int_{(i-1)/n}^{i/n}{\pmb\Theta}_t{\pmb\Theta}_t^T\, dt$ for $i=1,\ldots,n$, could be very different from the \ICV! We can easily change the $n$ covariance matrices $\int_{(i-1)/n}^{i/n}{\pmb\Theta}_t{\pmb\Theta}_t^T\, dt$ and hence the distributions of $\Delta{\bf X}_{i}$ \emph{without} changing \ICV. And as both the dimension $p$ and observation frequency $n$ go to infinity, there are just too much freedom in the underlying distributions which makes the inference about \ICV\  impossible. Certain structural assumptions are necessary to  turn the impossible into the possible. The simplest one is to assume ${\pmb\Theta}_t\equiv {\pmb\Theta}$, in which case $\Delta{\bf X}_{i}$ are \hbox{i.i.d.} The apparent shortcoming about this assumption is that it could not capture stochastic  volatility which is a stylized feature in financial data. The following class of processes, introduced in \cite{ZL11}, accommodates both stochastic  volatility  and leverage effect and in the meanwhile makes the inference about \ICV\ still possible (and as we will see soon that the theory is already much more complicated than the \hbox{i.i.d.} setting).

\begin{defn}\label{classC}
Suppose that $({\bf X}_t)$ is a $p$-dimensional process satisfying~(\ref{diffusion}).
We say that $({\bf X}_t)$ belongs to Class $\mathcal{C}$ if, almost surely,
there {exist} $(\ga_t)\in D([0,1];\bR)$
and $\pLa$ a $p\times p$ matrix
satisfying tr$(\pLa \pLa^T)=p$ such that
\begin{eqnarray}\label{coval:product}
\pTh_t=\gamma_t \ \pLa,
\end{eqnarray}
where $D([0,1];\bR)$ stands for the space of c\`{a}dl\`{a}g functions from $[0,1]$ to $\bR$.
\end{defn}

\begin{rmk}\label{rmk:ga_lam}
The convention that  tr$(\pLa \pLa^T)=p$  is made to
resolve the non-identifiability built in the formulation \eqref{coval:product}, in which
one can multiply $(\ga_{t})$ and divide $\pLa$ by a same constant
without modifying the process $(\pTh_t)$. It is thus not a restriction.
\end{rmk}

Class $\mathcal{C}$ incorporates some widely used models as special cases:
\begin{itemize}
    \item The simplest case is when the drift $(\pmu_t)\equiv 0$ and $(\ga_t)\equiv\ga$, in which case the returns $\De\X_i$ are \hbox{i.i.d.}  $ N(0,\ga^2/n\cdot\pLa \pLa^T)$.
    \item More generally, again when the drift $(\pmu_t)\equiv 0$ while $(\ga_t)$ is independent of the underlying Brownian motion $(\W_t)$, the returns $\De\X_i$ follow mixed normal distributions.
    \begin{itemize}
    \item Mixed normal distributions, or their asymptotic equivalent form in the high-dimensional setting\footnote{See Section 2 of \cite{ElKaroui13} for the asymptotic equivalence between the mixed normal distributions and the elliptic distributions in the high-dimensional setting.}, elliptic distributions have been widely used in financial applications. \cite{MFE05} state that ``elliptical distributions ... provided far superior models to the multivariate normal for daily and weekly US stock-return data'', and ``multivariate return data for groups of returns of a similar type often look roughly elliptical.''
      \item More recently, El Karoui in a series of papers (\cite{ElKaroui09}, \cite{ElKaroui10} and \cite{ElKaroui13}) studied the Markowitz optimization problem under the setting that the returns follow mixed normal/elliptic distributions.
    \end{itemize}
    \item Furthermore, Class $\mathcal{C}$ allows the drift $(\pmu_t)$ to be non-zero, and more importantly, the $(\gamma_t)$ process to be stochastic and even dependent on the Brownian motion $(\W_t)$ that drives the price process, thus featuring the so-called leverage effect in financial econometrics, which is an important stylized fact of financial returns and has drawn a lot of attention recent years, see, for example, \cite{AFL13} and \cite{WM14}.
\end{itemize}

Observe that if $\(\X_t\)$ belongs to Class $\mathcal{C}$, then the ICV matrix
\begin{eqnarray}
\ICV=\int_0^1 \gamma_t^2 \, dt \cdot\breve{\pSi}, ~~~~~~
{\rm where} ~~ \breve{\pSi}={\pLa}{\pLa}^T .
\label{ICV}
\end{eqnarray}
Furthermore, if the drift process ${\pmb \mu}_t\equiv 0$ and $(\gamma_t)$ is independent of $({\bf W}_t)$, then,
 conditional on $(\ga_t)$ and using \eqref{DelV}, we have
\[
\Delta\overline{{\bf X}}_{2i} \ \stackrel{d}{=} \
\sqrt{w_i}
\ \breve{{\pmb\Sigma}}^{1/2} \ {\bf Z}_i,
\]
where ${\bf Z}_i=(Z_i^{1},\ldots,Z_i^{p})^T$ consists of independent standard normals, and
\begin{eqnarray}
w_i&=&
\sum_{|j|<k} \left(1-\frac{|j|}{k}\right)^2 \int_{\frac{(2i-1)k+j-1}{n}}^{\frac{(2i-1)k+j}{n}} \ \gamma_t^2 \, dt.
\label{wie}
\end{eqnarray}
It follows that
\begin{eqnarray*}
\cA_m=\sum_{i=1}^m \Delta\overline{{\bf X}}_{2i} \cdot (\Delta\overline{{\bf X}}_{2i})^T & \stackrel{d}{=} & \sum_{i=1}^m \ w_i \ \breve{\pSi}^{1/2}  \Z_i  \Z_i^T  \breve{\pSi}^{1/2}.
\end{eqnarray*}

\subsubsection{Further inference based on $\cA_m$}\label{sssec:inf_PAV}
Using Corollary \ref{cor:A_PAV} and Theorem~1 in \cite{ZL11} we establish the following result concerning the LSD of $\cA_m$.

We put the following assumptions on the underlying process. They are inherited from Proposition 5 of \cite{ZL11}, and we refer the readers to that article for some further background and explanations. Observe in particular that Assumption \eqref{asm:leverage} allows the covolatility process to be dependent on the Brownian motion that drives the price processes. Such a dependence allows us to capture the so-called leverage effect in financial econometrics. Assumptions \eqref{asm:Sigma_bdd} and \eqref{asm:gamma_conv} are about the spectral norm of the \ICV\ matrix. We do not require the norm to be bounded, allowing, for example, spike eigenvalues.

\smallskip
\noindent{\bf Assumption  C:}
\begin{compactenum}\setcounter{enumi}{3}
 \item[]
 \begin{compactenum}
  \item\label{asm:X_in_C} For all $p$, $({\bf X}_t)$ is a $p$-dimensional process in Class $\mathcal{C}$ for some drift process ${\pmb\mu}_t=(\mu_t^{1},\ldots,\mu_t^{p})^T$ and covolatility process  $({\pmb\Theta}_t)=(\gamma_t{\pmb\Lambda})$;
  \item\label{asm:mu_bdd} there exists a $C_0<\infty$ such that for all $p$ and all $j=1,\ldots,p$, $|\mu_t^{j}|\le C_0$ for all $t\in [0,1)$ almost surely;
  \item\label{asm:Sigma_conv} as $p\to\infty$, the ESD of $\breve{{\pmb\Sigma}}={\pmb\Lambda}{\pmb\Lambda}^T$ converges in distribution to a probability distribution $\breve{H}$;
  \item\label{asm:Sigma_bdd} there exist $C_1<\infty$ and $\kappa<1/6$ such that for all $p$, $\|\breve{{\pmb\Sigma}}\|\le C_1p^{\kappa}$  almost surely;
  \item\label{asm:leverage} there exists  a sequence of index sets~$\mathcal{I}_p$ satisfying $\mathcal{I}_p\subset \{1,\ldots,p\}$ and $\#\mathcal{I}_p=o(p)$ such that $(\gamma_t)$ may depend on
      $({\bf W}_t)$ but only on $(W_t^{j}: j\in \mathcal{I}_p)$;
  \item\label{asm:gamma_conv}  there exists a $C_2<\infty$ such that  for all $p$ and for all $t\in [0,1)$,    $|\gamma_t|\le C_2$ almost surely, and additionally, almost surely,  $(\ga_t)$
  converges uniformly to a nonzero process $(\gamma_t^*)$ that is piecewise continuous with finitely many jumps.
\end{compactenum}
\end{compactenum}

\begin{thm}\label{thm:main}
Suppose that Assumptions \eqref{asm:X_in_C}-\eqref{asm:gamma_conv}  and \eqref{asm:k_PAV} hold, then as $p\to\infty$,
\begin{compactenum}[(i)]
\item the ESDs of \ICV\ and $\cA_m$  converge to probability distributions $H$ and $F^{\cA}$ respectively, where
\begin{equation}\label{Hx}
H(x) \ = \ \breve{H}(x/\zeta) \q \mbox{for  all }  x\geq 0 ~ {\rm with} ~ \zeta=\int_0^1(\gamma_t^*)^2\, dt;
\end{equation}
\item $F^{\cA}$ and $H$ are related as follows:
\begin{eqnarray}\label{mz_mainresult}
m_{\cA}(z)=-\dfrac{1}{z}\int\dfrac{\zeta}{\tau M(z)+\zeta} ~ dH(\tau),
\end{eqnarray}
where $M(z)$, together with another function $\wt{m}(z)$, uniquely solve the following equations in $\bC^+\times \bC^+$
\begin{equation}\label{Mm_z}
\left\{
\begin{array}{lll}
M(z) &=& -\dfrac{1}{z} \displaystyle\int_0^1 \dfrac{(1/3)(\ga_s^*)^2}{1+y\wt{m}(z) (1/3)(\ga_s^*)^2}ds,  \\
\wt{m}(z) &=& {-\dfrac{1}{z} \displaystyle\int \dfrac{\tau}{\tau M(z)+\zeta} dH(\tau).}
\end{array}
\right.
\end{equation}
\end{compactenum}
\end{thm}

\begin{rmk}\label{rmk:mainthm_usage}\
Based on Corollary \ref{cor:A_PAV} and Theorem \ref{thm:main}, we obtain the following two-step procedure to estimate the ESD of ICV:
\begin{figure}[H]
\begin{center}
\includegraphics[width=12cm]{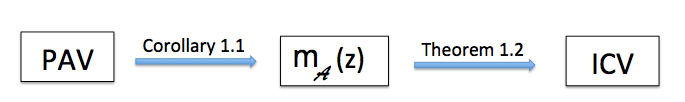}
\end{center}
\caption{A two-step procedure to estimate the ESD of ICV based on Corollary \ref{cor:A_PAV} and Theorem \ref{thm:main}.}
\label{fig:Step_1}
\end{figure}
\begin{compactenum}[(i)]
\item First, based on Corollary \ref{cor:A_PAV}, plug the ESD of PAV into equation \eqref{eqn:m_A} to solve for $m_\cA(z)$;
\item With the estimated $m_\cA(z)$ from the first step, use equations \eqref{mz_mainresult} and \eqref{Mm_z} in Theorem \ref{thm:main} to estimate the ESD of \ICV\  by generalizing the algorithms in \cite{ElKaroui08}, \cite{Mestre08}, and \cite{BCY10} \hbox{etc.}
\end{compactenum}
See the simulation studies for more detailed explanations about the estimation procedure.
\end{rmk}

\subsubsection{An alternative estimator}\label{sssec:B_n}
The aforementioned two-step procedure involves the $(\ga_s^*)$ which needs to be estimated in practice. Moreover, equations \eqref{mz_mainresult} and \eqref{Mm_z} involves three unknowns ($H, M(z)$ and $(\wt{m}(z))$) and the numerical solutions of these equations may be unstable. Motivated by this consideration, we draw ideas from \cite{ZL11} and propose an alternative estimator that overcomes these challenges. It is also worth mentioning that the alternative estimator allows for rather general dependence structures in the noise process, both cross-sectional and temporal. The temporal dependence between the microstructure noise has been documented in recent studies, see, for example, \cite{hansenlunde06}, \cite{UO09} and \cite{JLZ_noise_V2}.

Our alternative estimator is an extension of the time-variation adjusted RCV matrix introduced in \cite{ZL11} to our noisy setting.
To start, we fix an $\al\in(1/2,1)$ and $\th\in(0,\infty)$, and
let $k=\lfloor\th n^{\al}\rfloor$ and $m=\lfloor n/(2k)\rfloor$.
The time-variation adjusted PAV matrix is then defined as
\begin{equation}
\q \cB_m:= 3 \dfrac{\sum_{i=1}^m |\De\ol{\Y}_{2i}|^2}{m}\cdot\sum_{i=1}^m \dfrac{\De\ol{\Y}_{2i}(\De\ol{\Y}_{2i})^T}{|\De\ol{\Y}_{2i}|^2} \
= \ 3 \dfrac{\sum_{i=1}^m |\De\ol{\Y}_{2i}|^2}{p}\ \widetilde{\pSi},
\label{eqn:B_n}
\end{equation}
where
\begin{eqnarray}
\widetilde{\pSi}:= \ \dfrac{p}{m} \ \sum_{i=1}^m \dfrac{\De\ol{\Y}_{2i}(\De\ol{\Y}_{2i})^T}{|\De\ol{\Y}_{2i}|^2}.
\label{tSi}
\end{eqnarray}

Note that here window length $k$ has a higher order than  in Theorem~\ref{thm:main}.
The reason is that, after pre-averaging, the underlying returns are $O_p(\sqrt{k/n})$ and the noises are $O_p(\sqrt{1/k})$. In Theorem \ref{thm:main}, we balance the orders of the two terms by choosing $k=O(\sqrt{n})$ to achieve the optimal convergence rate.
In  Theorem \ref{thm:B_n} below,  we take $k=O(n^{\alpha})$ for some $\alpha>1/2$  to eliminate the impact of noise.

We first recall the concept of $\rho$-mixing coefficients.
\begin{defn} \label{rho_corr}
Suppose that $U=(U_k,k\in\zz{Z})$ is a stationary time series.  For $-\infty \le j\le \ell \le \infty$, let $\cF_j^\ell$ be the $\sigma$-field  generated by the random variables $(U_k: j\le k\le\ell)$.
The $\rho$-mixing coefficients  are defined as
\[
\rho(r) = \sup_{f\in\cL^2(\cF_{-\infty}^0), ~ g\in\cL^2(\cF_r^{\infty})} \ \left| {\rm Corr}(f,g)\right|, \q \mbox{for}\q r\in\zz{N},
\]
where for any probability space $\cD$,  $\cL^2(\cD)$ refers to the space of square-integrable, $\cD$-measurable random variables.
\end{defn}

We now introduce  a number of assumptions. Observe that Assumption~\eqref{asm:eps_general} says that we allow for rather general dependence structures in the noise process, both cross-sectional and temporal. We actually do not put any restrictions on the cross-sectional dependence, and even the dependence between the microstructure noise and the price process is allowed. Note also that \cite{JLZ_noise_V2} provides an approach to estimate the decay rate of the $\rho$-mixing coefficients.   Assumption \eqref{asm:leverage_2} is again about the dependence between the covolatility process and the Brownian motion that drives the price processes. Assumption \eqref{asm:vol_bdd} is about the boundedness of individual volatilities.

\begin{compactenum}\setcounter{enumi}{4}
 \item[]
 \begin{compactenum}
\item\label{asm:eps_general}  For all $j=1,\cdots,p$, the noises $(\vep_i^j)$ is stationary, have mean 0 and  bounded $4\ell$th moments and with $\rho$-mixing coefficients $\rho^j(r)$ satisfying  $\max_{j=1,\cdots,p}   \rho^j(r)=O(r^{-\ell})$ for some integer $\ell \geq 2$ as $r\to\infty$;
\item\label{asm:leverage_2} there exists a $0\le \de_1<1/2$ and a sequence of index sets $\mathcal{I}_p$ satisfying $\mathcal{I}_p\subset\{1,\ldots,p\}$ and $\# \mathcal{I}_p=O(p^{\de_1})$ such that $(\ga_t)$ may depend on $(\W_t)$ but only on $(W_t^{j}: j\in\mathcal{I}_p)$;
  \item\label{asm:gamma_bdd} there exists a $C_1<\infty$ such that for all $p$, $|\ga_t|\in\(1/C_1,C_1\)$ for all $t\in[0,1)$ almost surely;
  \item\label{asm:vol_bdd}  there exists a $C_2<\infty$ such that for all $p$ and  all $j$, the individual volatilities $\si_t=\sqrt{(\ga_t)^2\cdot\sum_{k=1}^p (\Lambda_{jk})^2}\in\(1/C_2,C_2\)$ for all $t\in[0,1]$ almost surely;
  \item\label{asm:Sigma_bdd_2} there exist $C_3<\infty$ and $0\le\de_2<1/2$ such that for all $p$, $\|\ICV\|\le C_3 p^{\de_2}$ almost surely;
  \item the $\de_1$ in \eqref{asm:leverage_2} and $\de_2$ in \eqref{asm:Sigma_bdd_2} satisfy that $\de_1+\de_2<1/2$;
  \item\label{asm:ym_conv}  $k=\lfloor\theta n^{\alpha}\rfloor$ for some $\th\in(0,\infty)$ and $\alpha\in [(3+\ell)/(2\ell+2) ,1)$,   and $m=\lfloor\frac{n}{2k}\rfloor$ satisfy that $\lim_{p\to\infty} p/m=y>0$,
  where $\ell$ is the integer in Assumption \eqref{asm:eps_general}.
 \end{compactenum}
\end{compactenum}

\begin{rmk}\label{rmk:compatibility} Careful readers may have noticed that Assumptions \eqref{asm:k_PAV} and \eqref{asm:ym_conv} are mathematically incompatible as  Assumption \eqref{asm:k_PAV} requires  $p=O(\sqrt{n})$ while  Assumption \eqref{asm:ym_conv} requires $p=O(n^{1-\al})$ for some $\al\in (1/2,1)$. The two assumptions are, however, perfectly compatible in practice when we deal with finite samples. In fact, take the choices of $(p,n,k)$ in the simulation studies (Section \ref{sec:simulation} below)  for example.
There we take $(p,n)=(100,\ 23,400)$. When applying Corollary \ref{cor:A_PAV} and Theorem \ref{thm:main}, we take $k=\lfloor0.5\sqrt{n}\rfloor=76$, which leads to $y=p/\lfloor n/2k\rfloor\approx 0.7$ in Assumption \eqref{asm:k_PAV}; while when applying Theorem \ref{thm:B_n} below, we take $k=\lfloor 1.5 n^{0.6}\rfloor=627$, which gives  $y=p/\lfloor n/2k\rfloor\approx 5.6$ in Assumption \eqref{asm:ym_conv}.
\end{rmk}

We have the following convergence result regarding the ESD of the alternative estimator.

\begin{thm}\label{thm:B_n}
Suppose that for all $p$, $({\bf X}_t)$ is a $p$-dimensional process in Class $\mathcal{C}$ for some drift process ${\pmb\mu}_t=(\mu_t^{1},\ldots,\mu_t^{p})^T$ and covolatility process $({\pmb\Theta}_t)= (\gamma_t{\pmb\Lambda}) $. Suppose also that Assumptions (\ref{asm:mu_bdd}), \eqref{asm:Sigma_conv} and \eqref{asm:gamma_conv}
 in Theorem \ref{thm:main} hold. Under Assumptions \eqref{asm:eps_general}-\eqref{asm:ym_conv}, we have as $p\to\infty$,
the ESDs of $\ICV$ and $\cB_m$ converge almost surely to probability distributions $H$ and
$F^\cB$, respectively, where $H$ satisfies \eqref{Hx}, and $F^\cB$ is determined by $H$ in that the Stieltjes transform of $F^\cB$, denoted by $m_\cB(z)$, satisfies the  following (standard) Mar\v{c}enko-Pastur equation
\begin{eqnarray}\label{eqn:mp}
~~~~ ~~  m_\cB(z)
=\int_{\tau\in\bR}\dfrac{1}{\tau\(1-y(1+z m_\cB(z))\)-z} \ dH(\tau),
~~{\rm for}~ z\in\bC^+.
\end{eqnarray}
\end{thm}

\begin{rmk}\label{rmk:thm3}
Theorem \ref{thm:B_n} says that the LSDs of ICV and $\cB_m$ are related via  the Mar\v{c}enko-Pastur equation.  Several algorithms have been developed to recover $H$ by inverting the Mar\v{c}enko-Pastur equation, see, for example, \cite{ElKaroui08, Mestre08, BCY10} \hbox{etc.}
We can therefore readily estimate the ESD of ICV by using these existing algorithms.
\end{rmk}

\begin{figure}[H]
\begin{center}
\centering
\includegraphics[width=9cm]{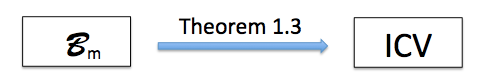}
\end{center}
\caption{Estimate the ESD of ICV based on Theorem \ref{thm:B_n}.}
\label{fig:Step_1}
\end{figure}

The rest of the paper is organized as follows. Section \ref{sec:simulation} demonstrates
how to implement both the two-step procedure introduced in Remark \ref{rmk:mainthm_usage} and the more direct method in Remark \ref{rmk:thm3} to estimate the spectral distribution of underlying covariance matrix based on noisy observations.
The proof of Theorem \ref{thm:LSD_signal_noise} is given in Section \ref{sec:proofs}.
Section \ref{sec:conclusion} concludes.
The proofs of some lemmas and Theorems \ref{thm:main} and \ref{thm:B_n} are given in the Appendix \ref{appendix:pfs}.  

\section{Simulation studies}\label{sec:simulation}

In this section, we illustrate how to make inferences using Corollary \ref{cor:A_PAV},
Theorems \ref{thm:main} and  \ref{thm:B_n}.
In particular, we will show how to estimate the ESD of ICV by using (1) Corollary \ref{cor:A_PAV} and Theorem~\ref{thm:main} based on PAV,  and (2) Theorem \ref{thm:B_n}  based on the alternative estimator  $\cB_m$.

We consider a stochastic U-shaped $(\ga_t)$ process  as follows:
\[
d\gamma_t=-\rho (\ga_t-\mu_t)\,dt +\sigma\, d\wt{W}_t, \q \mbox{for}\q t\in [0,1],
\]
where  $\rho=10$, $\sigma=0.05$,
\[
\mu_t \ = \ 2 \sqrt{0.0009+0.0008\cos(2\pi t)},
\]
and $\wt{W}_t=\sum_{i=1}^p W_t^{i}/ \sqrt{p}$ with $W_t^{i}$ being the $i$th component of the Brownian motion $(\W_t)$ that drives the price process. Observe that such a formulation makes $(\ga_t)$ to be dependent on \emph{all} the  component of the underlying Brownian motion, hence Assumptions \eqref{asm:leverage} and \eqref{asm:leverage_2} are actually both violated. However, we shall see soon that our methods still work well. A sample path of $(\ga_t)$ is given below.
\begin{figure}[H]
\begin{center}
\includegraphics[width=10cm]{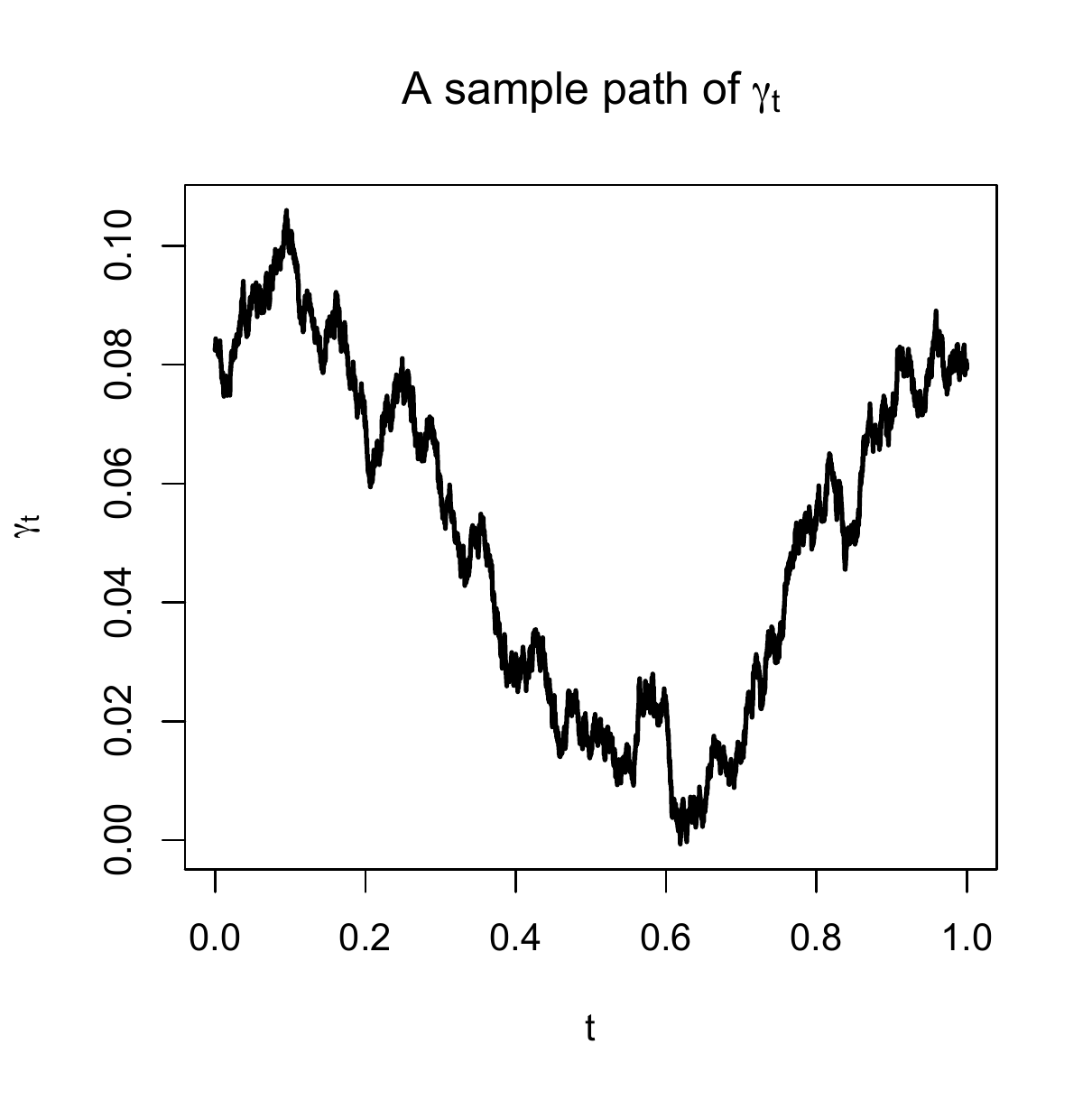}
\end{center}
\caption{A sample path of the process $(\ga_t)$.}
\label{fig:gamma_t}
\end{figure}
Next, the matrix $\breve{\pSi}$ is taken to be $UD U^T$ where $U$ is a random orthogonal matrix and $D$ is a diagonal matrix whose diagonal entries are drawn independently from Beta$(1,3)$ distribution.
With such  $(\ga_t)$ and $\breve{\pSi}$, the individual daily volatilities are around $3\%$, which is similar to what one observes in practice.
The latent log price process $(\X_t)$ follows
\[
d\X_t =\ga_t  \pLa\, d\W_t, \q \mbox{where}\q \pLa =\breve{\pSi}^{1/2}.
\]
Finally, the noise $(\pvep_i)_{1\le i\le n}$ are drawn from \hbox{i.i.d.} $N(0,0.0002\, \I)$.

In the studies below, the dimension, i.e., the number of assets $p$ is taken to be 100, and the observation frequency $n$ is set to be 23,400 which corresponds to one observation per second on a regular trading day.

\subsection{Applications of Corollary \ref{cor:A_PAV} and Theorem \ref{thm:main}} \label{subsec:app_Cor_Thm2}
In this subsection, we illustrate how to estimate the ESD of ICV based on PAV matrix by using the two-step procedure that we introduced in Remark \ref{rmk:mainthm_usage}.

In the first step, we replace $F$  in equation \eqref{eqn:m_A} with the ESD of PAV and solve for $m_\cA(z)$. The window length $k$ in defining PAV is set to be $\lfloor 0.5\sqrt{n}\rfloor$. As to $m_\cA(z)$ to be solved, we choose a set of  $z$'s whose real and imaginary parts are equally spaced in the intervals $[-20,0]$ and $[1,20]$ respectively. Denote these $z$'s by $\{z_j\}_{j=1}^J$, and the estimated $m_\cA(z_j)$ by $\wh{m_\cA(z_j)}$.
We then need to estimate the ESD of $\cA_m$ based on $\{\wh{m_\cA(z_j)}\}_{j=1}^J$, which is done as follows.

Inspired by the nonparametric estimation method proposed in \cite{ElKaroui08}, we approximate $F^{\cA_m}$ with a weighted sum of point masses
\begin{equation}\label{eq:disc_cA}
F^{\cA_m} \approx \sum_{k=1}^K w_k \de_{x_k},
\end{equation}
where $\{x_1 < x_2< \ldots< x_K\}$ is a grid of points to be specified, and $w_k$'s are weights to be estimated.
To choose the grid $\{x_k\}_{k=1}^K$,
naturally we would like $[x_1, x_K]$ to cover the support of $F^{\cA_m}$ which, however, is unknown. To overcome such a difficulty, note that by the Mar\v{c}enko-Pastur theorem, the support of ESD of sample covariance matrix always covers that of the population covariance matrix, and since by Theorem \ref{thm:B_n}, our alternative estimator, $\cB_m$, satisfies the same Mar\v{c}enko-Pastur equation as the sample covariance matrix, $\cB_m$ inherits such a property with a support covering that of \ICV. (Such a feature can be clearly seen in the first plot in Figure \ref{fig:Thm3}.) Thanks to this property,  we choose $x_k$'s be equally spaced between 0 and the largest eigenvalue of $\cB_m$, and we are guaranteed that $[x_1, x_K]$ will cover the support of $F^{\cA_m}$.

Next we discuss how to estimate the weights $\{w_k\}$. Observe that the discretization  \eqref{eq:disc_cA}  gives an approximated Stieltjes transform of $F^{\cA_m}$ as $\sum_{k=1}^K \dfrac{w_k}{x_k-z}.$
Let
\[
e_j^{'}:=\wh{m_\cA(z_j)} -  \sum_{k=1}^K \dfrac{w_k}{x_k-z_j}, \q j=1,\cdots, J
\]
be  the approximation errors.
The weights $\{w_k\}_{k=1}^K$ are then estimated by minimizing the approximation errors:
\begin{equation}\label{eq:est_wt_A}
\argmin\limits_{(w_1,\ldots,w_k)}\max\limits_{j=1,2,\cdots,J} \max \{|\Re(e_j)|, |\Im(e_j)|\}\quad \text{subject}\ \text{to} \quad \sum_{k=1}^K w_k =1 \mbox{ and } w_k\geq 0.
\end{equation}

We move on the estimation of the ESD of ICV by using Theorem \ref{thm:main}. We first need to estimate two unknowns, $M(z)$
and $\wt{m}(z)$, which we do as follows. First note that multiplying $\wt{m}(z)$ and $M(z)$ on both sides of the first
and second equations in \eqref{Mm_z} respectively yields
\[
M(z)\cdot \wt{m}(z) = -\dfrac{1}{yz} + \dfrac{1}{yz} \int_0^1 \dfrac{1}{1+y\wt{m}(z)(1/3)(\ga_s^*)^2} ~ ds,
\]
and
\[
M(z)\cdot \wt{m}(z) =  -\dfrac{1}{z} \int \dfrac{\tau M(z)}{\tau M(z)+\zeta} ~ dH(\tau) = -\dfrac1z -m_\cA(z),
\]
where the last step is due to equation \eqref{mz_mainresult}. It follows that
\begin{equation}\label{Mm_1_2}
-\dfrac1z -m_{\cA}(z) = -\dfrac{1}{yz} +\dfrac{1}{yz} \int_0^1 \dfrac{1}{1+y\wt{m}(z)(1/3)(\ga_s^*)^2}~ ds,
\end{equation}
and $\wt{m}(z)=-(1/z+m_\cA(z))/M(z).$
Substituting the last expression of $\wt{m}(z)$ into equation \eqref{Mm_1_2} yields
\begin{eqnarray}\label{eqn:sol_Mz}
\int_0^1\dfrac{M(z)}{M(z)-(1/3)(\ga_s^*)^2 y(z^{-1}+m_\cA(z))} ~ ds =1-y-yzm_\cA(z).
\end{eqnarray}
By plugging the $\{\wh{m_\cA(z_j)}\}_{j=1}^J$ obtained in the first step and solving equation \eqref{eqn:sol_Mz} we get $\{\wh{M(z_j)}\}_{j=1}^J$.

We are now ready to estimate the ESD of \ICV. Similarly as in the first step, discretize $F^{\ICV}$ as
\begin{eqnarray}\label{dH_ICV}
F^{\ICV} \approx \sum_{k=1}^K c_k \de_{x_k},
\end{eqnarray}
where $c_k$'s are weights to be estimated. By equation \eqref{mz_mainresult} we expect that
\[
e_j^{''}:= \wh{m_\cA(z_j)} + \dfrac{1}{z_j} \cdot \sum_{k=1}^K ~ c_k~ \dfrac{\zeta}{x_k \wh{M(z_j)} +\zeta}
\]
to be small. The $c_k$'s are then estimated by minimizing the approximation errors $e_j^{''}$ just as in \eqref{eq:est_wt_A}.

Figure \ref{fig:Thm1-2} below illustrates the estimation procedure. The left plot shows three ESDs, those of \ICV, $\cA_m$ and PAV respectively. The three curves are clearly different from each. Keep in mind that we only observe that of PAV, whereas the ESDs of both \ICV\ and $\cA_m$ are underlying and need to estimated. The estimation of the ESD of $\cA_m$ is conducted in the first step, and the result is shown in the middle plot. The second step estimates the ESD of \ICV, given in the right plot. We see in both plots that the estimated ESDs roughly match with their respective targets, showing that our proposed two-step procedure indeed works in practice.

\begin{figure}[H]
\begin{center}
\includegraphics[width=13cm]{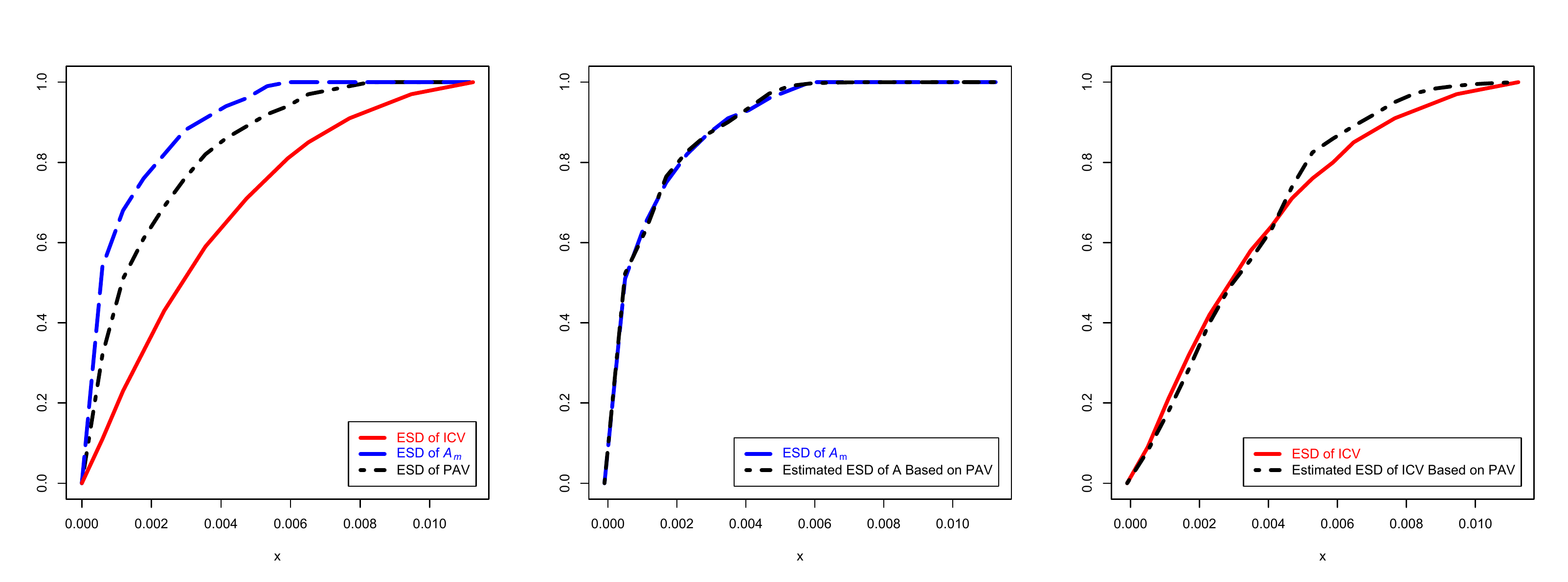}
\end{center}
\caption{Estimation of the ESDs of $\cA_m$ and \ICV\ based on  Corollary \ref{cor:A_PAV} and Theorem \ref{thm:main}.}
\label{fig:Thm1-2}
\end{figure}

\subsection{Application of Theorem \ref{thm:B_n}} \label{subsec:simulation_thm3}

In this subsection we illustrate how to use our alternative estimator $\cB_m$ and Theorem \ref{thm:B_n} to estimate the ESD of \ICV. According to Theorem \ref{thm:B_n}, asymptotically, the ESD of  $\cB_m$ is related to that of ICV through the standard Mar\v{c}enko-Pastur equation. Several algorithms have been developed to estimate the spectra of the population covariance matrices by inverting the Mar\v{c}enko-Pastur equation, and in the below we adopt the algorithm proposed in \cite{ElKaroui08}.
Set the window length $k$ in defining $\cB_m$ to be $\lfloor1.5n^{0.6}\rfloor$. Discretize the ESD of ICV as \eqref{dH_ICV}.
According to Theorem \ref{thm:B_n}, the Stieltjes transform of the ESD of $\cB_m$, denoted by $m_{\cB_m}(z)$, should approximately satisfy equation \eqref{eqn:mp} with $H$ replaced with the ESD of ICV. In other words, we again expect the approximation errors
\[
e_j^{'''}:= m_{\cB_m}(z_j) -\sum_{k=1}^K \dfrac{c_k}{x_k(1-y(1+z_jm_{\cB_m}(z_j)))-z_j}
\]
to be small. So again, we estimate the weights $c_k$'s by minimizing the approximation errors $e_j^{'''}$ as in \eqref{eq:est_wt_A}.

The estimation results are given in Figure \ref{fig:Thm3}. Observe that the left plot shows clearly that the ESD of $\cB_m$ differs from the (latent unobserved) ESD of  \ICV, yet the right plot shows that we can estimate this latent distribution.

\begin{figure}[H]
\begin{center}
\includegraphics[width=13cm]{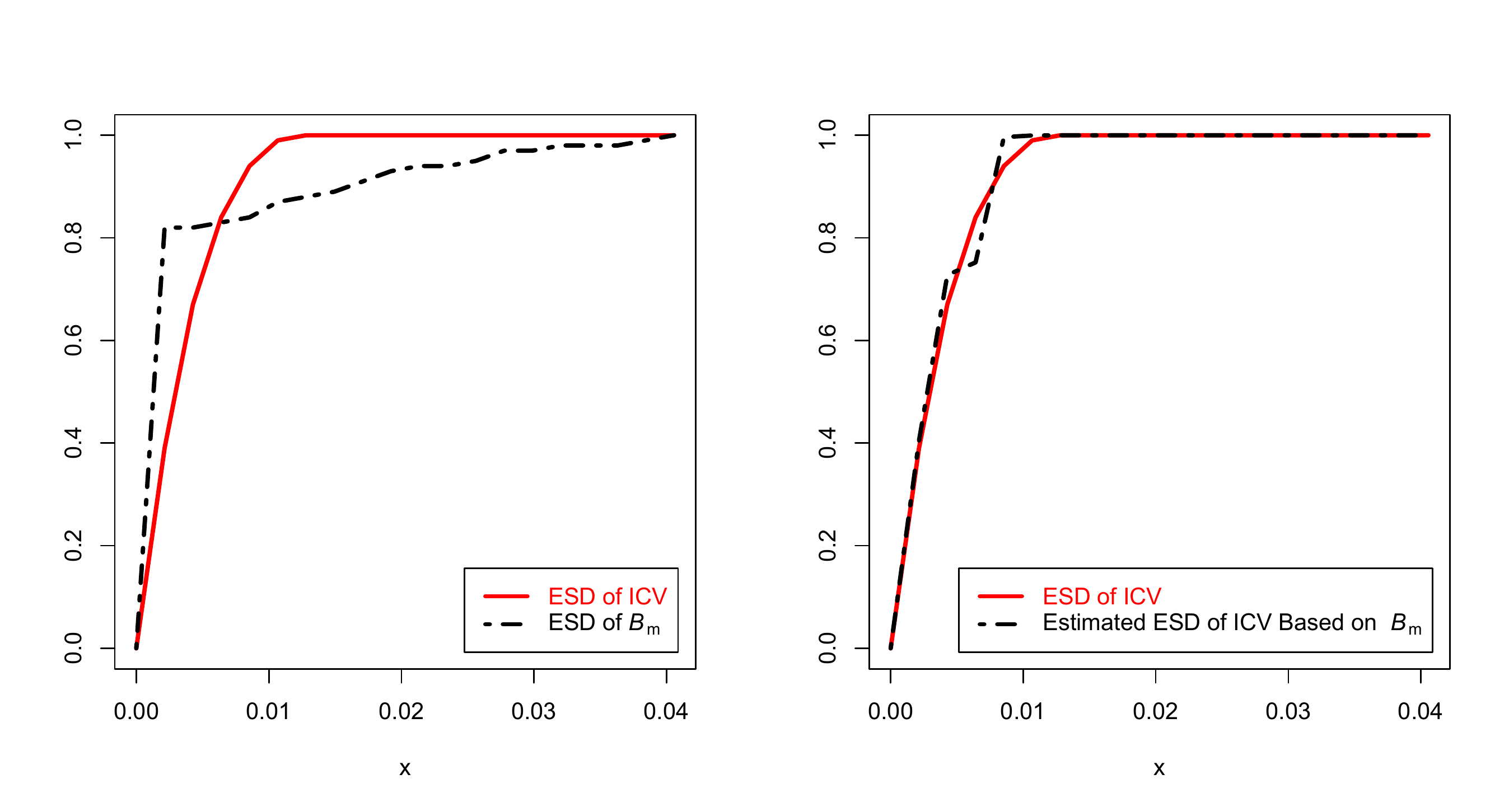}
\end{center}
\caption{Estimation of the ESDs of \ICV\ based on  Theorem \ref{thm:B_n}.}
\label{fig:Thm3}
\end{figure}

\section{Proofs}\label{sec:proofs}

\subsection{Proof of Theorem \ref{thm:LSD_signal_noise}}\label{sec:pf_LSD_signal_noise}
Theorem \ref{thm:LSD_signal_noise} is a consequence of the following proposition.

\begin{prop}\label{prop:LSD_signal_noise}
Under the assumptions of Theorem \ref{thm:LSD_signal_noise}, there exists a constant $K^*>0$ such that almost surely, for all
\[
z\in\mathbb{C}^{*}:=\left\{z\in\mathbb{C}^+: \Im(z)> K^*\right\},
\]
we have
\begin{equation}\label{eqn:prop_LSD_signal_noise}
\lim_{p\to\infty}\left[
\dfrac{1}{p}{\rm tr}\left(\dfrac{1}{1+\delta_n}\cA_n-z\I\right)^{-1}
-\dfrac{1}{p}{\rm tr}\bigg({\bf S}_n-(z-t_n\sigma_n^2)\I\bigg)^{-1}
\right]=0,
\end{equation}
where for all $p$ large enough,  $t_n$ is the unique solution to the equation
\begin{equation}\label{eqn:t_n}
t_n \ = \ y_n-1+y_n(z-t_n\si_n^2) \ \dfrac{1}{p}{\rm tr}\(\S_n-(z-t_n\si_n^2)\I \)^{-1},
\end{equation}
in the set
\begin{equation}\label{dfn:D}
\sD:=\left\{t\in\bC: 0\le\Im(t)\le \dfrac{\Im(z)}{2(\si+1)^2}\right\},
\end{equation}
 and
\begin{equation}
\de_n \ = \ y_n\si_n^2 \ \dfrac{1}{p}{\rm tr}\(\S_n-(z-t_n\si_n^2) \I \)^{-1}.
\label{delta}
\end{equation}
\end{prop}

The proof of Proposition \ref{prop:LSD_signal_noise} is given in Section \ref{ssec:pf_prop_LSD_signal_noise} after some preparation works have been  done in Sections \ref{ssec:t_n} and \ref{ssec:prelim}. In Section \ref{ssec:thm_LSD_signal_noise} we show how to establish Theorem \ref{thm:LSD_signal_noise} based on Proposition \ref{prop:LSD_signal_noise}.

To prove Proposition \ref{prop:LSD_signal_noise}, we shall use the following results from  \cite{DS2007a}.
By Theorem 1.1 therein, the sequence $\{F^{\S_n}\}$ converges weakly to a probability distribution  $F$.  Moreover, by using the same truncation and centralization technique as in \cite{DS2007a}, we may assume that
\begin{compactenum}\setcounter{enumi}{5}
 \item[]
 \begin{compactenum}
  \item\label{asm:eps_bd} $|\ep_{11}|\le a\log(n)$ for some $a>2$,
  \item\label{asm:eps_mean_var} $E\ep_{11}=0$, $E|\ep_{11}|^2=1$, \q and
  \item\label{asm:A_bd}  $\|(1/n)\A_n\A_n^T\|\le \log(n)$.\\
\end{compactenum}
\end{compactenum}

In addition to equation \eqref{eqn:t_n}, we shall also study its limiting equation
\begin{equation}\label{eqn:t}
t= \ y-1+y(z-t\sigma^2)m(z-t\sigma^2),
\end{equation}
where $m(\cdot)$ is  the Stieltjes transform of  the probability distribution  $F$.

It is shown in \cite{DS2007b} that the distribution $F$ admits a continuous density on $\zz{R}\setminus\{0\}$.  Since we assume that
$F$ is supported by a finite interval $[a,b]$ with $a>0$ and possibly has a point mass at 0, we conclude that $F$ admits a bounded density~$f$ supported by a finite interval $[a,b]$ and possibly a point mass at zero.

\subsubsection{Properties of $t_n$ and $t$}\label{ssec:t_n}

\begin{lem}\label{lem:tn_exist}
There exists a constant $K_1>0$ such that for all
$z\in\bC_1:=\left\{z=u+iv:  v> K_1 \right\}$,
for all $n$ large enough, equation (\ref{eqn:t_n}) admits a unique solution in $\sD$.
\end{lem}

\begin{lem}\label{lem:t_properties}
Suppose that $t$ solves equation \eqref{eqn:t} for $z\in\bC^+$. Write $t=t_1+it_2$ and $z=u+iv$.
Then $0<t_2<v/\si^2$; moreover,
as $v\to\infty$, uniformly in $u$, one has $t_2\to 0$ and $t_1\to -1$.
\end{lem}

\begin{lem}\label{lem:t_exist}
There exists a constant $K_2\geq K_1$ such that for any
$z\in\bC_2:=\left\{z=u+iv: v>K_2\right\}$, equation \eqref{eqn:t} admits a unique solution.
\end{lem}


\begin{lem}\label{lem:t_analytic}
There exists a constant $K_3\geq K_2$ such that the solution $t=t(z)$ to \eqref{eqn:t} is analytic on $\bC_3:=\left\{z=u+iv: v>K_3\right\}$.
\end{lem}


\begin{lem}\label{lem:tlimit}
Suppose that $t_n$ solves equation \eqref{eqn:t_n} for $z\in\bC_2$, then $\Im(t_n)>0$ and $\Im(z-t_n\si_n^2)>0$; moreover if $t_n$ is the unique solution in the  set $\sD$, then  with probability one, as $n\to\infty$, $t_n$ converges to a nonrandom complex number $t$ which uniquely solves equation \eqref{eqn:t}.
\end{lem}

The proofs of Lemmas   \ref{lem:tn_exist} - \ref{lem:tlimit} 
are given in the Appendix \ref{appendix:pfs}. 


\subsubsection{Some further preliminary results}\label{ssec:prelim}
Let $K^*=\max\{K_1,K_2,K_3\}$~$(~=K_3)$ for  $K_1$, $K_2$ and $K_3$ defined in Lemmas \ref{lem:tn_exist}, \ref{lem:t_exist} and \ref{lem:t_analytic}, respectively. And define
$\bC^*=\{z\in\bC^+: \Im(z)>K^*\}$.
Below we work with $z\in\bC^*$.

Let ${\bf a}_j$ and $\pep_j$, $j=1,\ldots,n$, be the $j$th column of ${\bf A}_n$ and  $\pvep_n$, and let ${\bf b}_j=\sigma_n \pep_j$.
Denote ${\pmb {\pmb\xi}}_j=\dfrac{1}{\sqrt{n}}({\bf a}_j+{\bf b}_j)$ so that ${\bf S}_n=\sum_{j=1}^n{\pmb {\pmb\xi}}_j{\pmb {\pmb\xi}}_j^T$.
For any complex number $t_n$ such that $\Im(z-t_n\si_n^2)> 0$,  define
\begin{eqnarray}
&&~~~{\bf R}_n  =  {\bf S}_n-(z-t_n\sigma_n^2)  \I ,\q\q\q\q \delta_n =\dfrac{\sigma_n^2}{n}{\rm tr}({\bf R}_n^{-1})
=y_n\si_n^2\dfrac{1}{p}{\rm tr}(\R_n^{-1}), \nonumber\\
&&~~~{\bf S}_{nj}  =  {\bf S}_n-{\pmb \xi}_j{\pmb \xi}_j^T \ = \ \sum_{k\neq j}{\pmb\xi}_k{\pmb\xi}_k^T , ~~\q
{\bf R}_{nj}  =  {\bf S}_{nj}-(z-t_n\sigma_n^2)  \I, \label{betaj}\\
&&~~~{\bf B}_n  =  \dfrac{1}{1+\delta_n}\dfrac{1}{n}{\bf A}_n{\bf A}_n^T-z  \I, \mbox{ and }
~ \beta_j  =  \dfrac{1}{ \ 1+{\pmb\xi}_j^T{\bf R}_{nj}^{-1}{\pmb\xi}_j \ }. \nonumber
\end{eqnarray}

According to equation (2.2) in \cite{SB95}, we have
\begin{eqnarray}\label{xiRn}
{\pmb\xi}_j^T \ \R_n^{-1}
= \ \dfrac{{\pmb\xi}_j^T{\bf R}_{nj}^{-1}}{ \ 1+{\pmb\xi}_j^T{\bf R}_{nj}^{-1}{\pmb\xi}_j \ }=\beta_j{\pmb\xi_j}^T\R_{nj}^{-1}.
\end{eqnarray}
Thus  using the identity
$
\A^{-1}-\B^{-1}=\A^{-1}(\B-\A)\B^{-1},
$
we obtain that
\begin{eqnarray}
{\bf R}_n^{-1} \ = \ {\bf R}_{nj}^{-1}-{\bf R}_n^{-1}{\pmb\xi}_j{\pmb\xi}_j^T{\bf R}_{nj}^{-1} \
= \ {\bf R}_{nj}^{-1}-\beta_j{\bf R}_{nj}^{-1}{\pmb\xi}_j{\pmb\xi}_j^T{\bf R}_{nj}^{-1}.
\label{RRbeta}
\end{eqnarray}

Next we introduce another definition of $t_n$, as the solution to the following equation
\begin{eqnarray}\label{dfn:t_alternative}
t_n = \ -\dfrac{1}{n}\sum_{j=1}^n \beta_j \
= \ -\dfrac{1}{n}\sum_{j=1}^n \ \dfrac{1}{ \ 1+{\pmb\xi}_j^T{\bf R}_{nj}^{-1}{\pmb\xi}_j \ }.
\end{eqnarray}
We claim that the definition of $t_n$ in (\ref{dfn:t_alternative}) is equivalent to the earlier definition of defining $t_n$ to be the solution to equation \eqref{eqn:t_n}.
In fact, write
\[
\R_n+z  \I \ = \ \sum_{j=1}^n \ {\pmb\xi}_j{\pmb\xi}_j^T+t_n\sigma_n^2 \ \I.
\]
Right-multiplying both sides by ${\bf R}_n^{-1}$ and using \eqref{xiRn} yield
\[
\I+z \ {\bf R}_n^{-1}=\sum_{j=1}^n \ {\pmb\xi}_j{\pmb\xi}_j^T{\bf R}_n^{-1}+t_n\sigma_n^2 \ {\bf R}_n^{-1}
=\sum_{j=1}^n \ \dfrac{{\pmb\xi}_j{\pmb\xi}_j^T{\bf R}_{nj}^{-1}}{ \ 1+{\pmb\xi}_j^T{\bf R}_{nj}^{-1}{\pmb\xi}_j \ } +t_n\sigma_n^2 \
 {\bf R}_n^{-1}.
\]
Taking  trace on both sides and dividing by $n$ one gets that
\begin{eqnarray}\label{beta_tn}
y_n+z \ \dfrac{1}{n}{\rm tr}({\bf R}_n^{-1})
&=& 1-\dfrac{1}{n}\sum_{j=1}^n
\dfrac{1}{ \ 1+{\pmb\xi}_j^T{\bf R}_{nj}^{-1}{\pmb\xi}_j \ }
+t_n\sigma_n^2 \ \dfrac{1}{n}{\rm tr}({\bf R}_n^{-1}) \nonumber\\
&=&1-\dfrac{1}{n}\sum_{j=1}^n \beta_j
+t_n\sigma_n^2 \ \dfrac{1}{n}{\rm tr}({\bf R}_n^{-1}).
\end{eqnarray}
This shows that if $t_n$ satisfies \eqref{dfn:t_alternative}, then $t_n$ satisfies equation \eqref{eqn:t_n}.
On the other hand, if $t_n$ satisfies equation \eqref{eqn:t_n}, from \eqref{beta_tn} we have
\begin{eqnarray*}
-\dfrac{1}{n}\sum_{j=1}^n \beta_j
&=& y_n-1+(z-t_n\si_n^2)\dfrac{1}{n}\tr(\R_n)^{-1} \\
&=& t_n,
\end{eqnarray*}
namely, $t_n$ satisfies \eqref{dfn:t_alternative}.

We proceed to analyze the difference in \eqref{eqn:prop_LSD_signal_noise}.
Since
\begin{eqnarray*}
{\bf S}_n-\dfrac{1}{1+\delta_n} \ \dfrac{1}{n}{\bf A}_n{\bf A}_n^T
&=&\dfrac{1}{n}\sum_{j=1}^n \ ({\bf a}_j+{\bf b}_j)({\bf a}_j+{\bf b}_j)^T-\dfrac{1}{1+\delta_n} \ \dfrac{1}{n}\sum_{j=1}^n  {\bf a}_j{\bf a}_j^T\\
&=&\dfrac{1}{n}\sum_{j=1}^n \ \left(\dfrac{\delta_n}{1+\delta_n} \ {\bf a}_j{\bf a}_j^T+{\bf a}_j{\bf b}_j^T+{\bf b}_j{\bf a}_j^T+{\bf b}_j{\bf b}_j^T\right),
\end{eqnarray*}
we have
\begin{eqnarray*}
\Delta&:=&\dfrac{1}{p}{\rm tr}\left[\(\dfrac{1}{1+\delta_n}\dfrac{1}{n}{\bf A}_n{\bf A}_n^T-z\I\)^{-1}
-\left({\bf S}_n-(z-t_n\sigma_n^2)\I\right)^{-1}\right]\\
&=&\dfrac{1}{p}{\rm tr}\(\(\dfrac{1}{1+\delta_n}\dfrac{1}{n}{\bf A}_n{\bf A}_n^T-z\I\)^{-1}
\left({\bf S}_n-\dfrac{1}{1+\delta_n}\dfrac{1}{n}{\bf A}_n{\bf A}_n^T+t_n\sigma_n^2 \I\right) \right.\\
&& ~~~~~~~~~~~~~~~~~~~~~~~~~~~~~~~~~~~~~~~~~~~ \times
\left.\left({\bf S}_n-(z-t_n\sigma_n^2)\I\right)^{-1}\)
\end{eqnarray*}
\begin{eqnarray*}
&=&\dfrac{1}{np}\sum_{j=1}^n \ \Bigg\{
\dfrac{\delta_n}{1+\delta_n} \ {\bf a}_j^T\left({\bf S}_n-(z-t_n\sigma_n^2)\I\right)^{-1}\left(\dfrac{1}{n(1+\delta_n)}{\bf A}_n{\bf A}_n^T-z\I\right)^{-1}{\bf a}_j\\
&&+ \ {\bf b}_j^T\left({\bf S}_n-(z-t_n\sigma_n^2)\I\right)^{-1}\left(\dfrac{1}{n(1+\delta_n)}{\bf A}_n{\bf A}_n^T-z\I\right)^{-1}{\bf a}_j\\
&&+ \ {\bf a}_j^T\left({\bf S}_n-(z-t_n\sigma_n^2)\I\right)^{-1}\left(\dfrac{1}{n(1+\delta_n)}{\bf A}_n{\bf A}_n^T-z\I\right)^{-1}{\bf b}_j\\
&&+ \ {\bf b}_j^T\left({\bf S}_n-(z-t_n\sigma_n^2)\I\right)^{-1}\left(\dfrac{1}{n(1+\delta_n)}{\bf A}_n{\bf A}_n^T-z\I\right)^{-1}{\bf b}_j\Bigg\}\\
&&+ \ \dfrac{t_n\sigma_n^2}{p}{\rm tr}\left(\left({\bf S}_n-(z-t_n\sigma_n^2)\I\right)^{-1}\left(\dfrac{1}{n(1+\delta_n)}{\bf A}_n{\bf A}_n^T-z\I\right)^{-1}\right).
\end{eqnarray*}

Recall the definitions of $\R_n$, $\R_{nj}$, ${\bf B}_n$ and $\beta_j$ in (\ref{betaj}). Using (\ref{RRbeta}) we have
\begin{eqnarray*}
\Delta&=&\dfrac{1}{np} \ \sum_{j=1}^n \ \bigg[ \ \dfrac{\delta_n}{1+\delta_n} \ {\bf a}_j^T{\bf R}_{nj}^{-1}{\bf B}_n^{-1}{\bf a}_j \ - \ \dfrac{\delta_n}{1+\delta_n} \ \beta_j \  {\bf a}_j^T{\bf R}_{nj}^{-1}{\pmb\xi}_j{\pmb\xi}_j^T{\bf R}_{nj}^{-1}{\bf B}_n^{-1}{\bf a}_j\\
&&+ \ {\bf b}_j^T{\bf R}_{nj}^{-1}{\bf B}_n^{-1}{\bf a}_j \ - \ \beta_j \ {\bf b}_j^T{\bf R}_{nj}^{-1}{\pmb\xi}_j{\pmb\xi}_j^T{\bf R}_{nj}^{-1}{\bf B}_n^{-1}{\bf a}_j\\
&&  + \ {\bf a}_j^T{\bf R}_{nj}^{-1}{\bf B}_n^{-1}{\bf b}_j \ - \ \beta_j \ {\bf a}_j^T{\bf R}_{nj}^{-1}{\pmb\xi}_j{\pmb\xi}_j^T{\bf R}_{nj}^{-1}{\bf B}_n^{-1}{\bf b}_j\\
&&+ \ {\bf b}_j^T{\bf R}_{nj}^{-1}{\bf B}_n^{-1}{\bf b}_j \ - \ \beta_j \ {\bf b}_j^T{\bf R}_{nj}^{-1}{\pmb\xi}_j{\pmb\xi}_j^T{\bf R}_{nj}^{-1}{\bf B}_n^{-1}{\bf b}_j \ \bigg]\\
&&+ \ \dfrac{t_n\sigma_n^2}{p} \ {\rm tr}({\bf R}_n^{-1}{\bf B}_n^{-1}).
\end{eqnarray*}
Define
\begin{eqnarray}\label{etaj}
\begin{array}{cc}
\rho_j=\dfrac{1}{n}{\bf a}_j^T{\bf R}_{nj}^{-1}{\bf a}_j,
& \hat{\rho}_j=\dfrac{1}{n}{\bf a}_j^T{\bf R}_{nj}^{-1}{\bf B}_n^{-1}{\bf a}_j, \\
w_j=\dfrac{1}{n}{\bf b}_j^T{\bf R}_{nj}^{-1}{\bf b}_j,
& \hat{w}_j=\dfrac{1}{n}{\bf b}_j^T{\bf R}_{nj}^{-1}{\bf B}_n^{-1}{\bf b}_j, \\
\eta_j=\dfrac{1}{n}{\bf a}_j^T{\bf R}_{nj}^{-1}{\bf b}_j,
& \hat{\eta}_j=\dfrac{1}{n}{\bf a}_j^T{\bf R}_{nj}^{-1}{\bf B}_n^{-1}{\bf b}_j,\\
\gamma_j=\dfrac{1}{n}{\bf b}_j^T{\bf R}_{nj}^{-1}{\bf a}_j,
&\hat{\gamma}_j=\dfrac{1}{n}{\bf b}_j^T{\bf R}_{nj}^{-1}{\bf B}_n^{-1}{\bf a}_j.
\end{array}
\end{eqnarray}
Certainly $\eta_j=\ga_j$, but introducing $\ga_j$ makes the computations below more clear.

Recall that ${\pmb\xi}_j=(1/\sqrt{n})({\bf a}_j+{\bf b}_j)$, and so
$\beta_j^{-1}=1+\rho_j+w_j+\eta_j+\ga_j$.
We can then rewrite $\Delta$ as
\begin{eqnarray*}
\Delta&=&\dfrac{1}{p}\sum_{j=1}^n\beta_j\bigg(
\dfrac{\delta_n}{1+\delta_n}\hat{\rho}_j(1+\rho_j+\eta_j+\gamma_j+w_j)
-\dfrac{\delta_n}{1+\delta_n}(\rho_j+\eta_j)(\hat{\rho}_j+\hat{\ga}_j)\\
&&+\hat{\gamma}_j(1+\rho_j+\eta_j+\gamma_j+w_j)-(\ga_j+w_j)(\hat{\ga}_j+\hat{\rho}_j)\\
&&+\hat{\eta}_j(1+\rho_j+\eta_j+\gamma_j+w_j)-(\rho_j+\eta_j)(\hat{\eta}_j+\hat{w}_j)\\
&&+\hat{w}_j(1+\rho_j+\eta_j+\gamma_j+w_j)-(\gamma_j+w_j)(\hat{\eta}_j+\hat{w}_j)\bigg)\\
&&+\dfrac{t_n\sigma_n^2}{p} {\rm tr}({\bf R}_n^{-1}{\bf B}_n^{-1})\\
&=&\dfrac{1}{p}\sum_{j=1}^n \beta_j\left(\dfrac{1}{1+\delta_n}\hat{\rho}_j(\delta_n-\gamma_j-w_j)+
\hat{\gamma}_j\left(1+\dfrac{1}{1+\delta_n}(\rho_j+\eta_j)\right)+\hat{\eta}_j+\hat{w}_j\right)\\
&&+\dfrac{t_n\sigma_n^2}{p} {\rm tr}({\bf R}_n^{-1}{\bf B}_n^{-1})\\
&:=&\Delta_1+\Delta_2+\Delta_3,
\end{eqnarray*}
where
\begin{align}\label{eq:Delta}
\notag \Delta_1
\notag         =&\dfrac{1}{p(1+\delta_n)}\sum_{j=1}^n\beta_j\hat{\rho}_j(\delta_n-w_j)
-\dfrac{1}{p(1+\delta_n)}\sum_{j=1}^n\beta_j\hat{\rho}_j\gamma_j,\\
      \Delta_2
       =&\dfrac{1}{p}\sum_{j=1}^n\beta_j\hat{\gamma}_j\left(1+\dfrac{1}{1+\delta_n}(\rho_j+\eta_j)\right)
+\dfrac{1}{p}\sum_{j=1}^n\beta_j\hat{\eta}_j, \q \mbox{and}\\
\notag  \Delta_3
\notag         =&\dfrac{1}{p}\sum_{j=1}^n\beta_j\left(\hat{w}_j-\dfrac{\sigma_n^2}{n} {\rm tr}({\bf R}_n^{-1}{\bf B}_n^{-1})\right),
\end{align}
where in the last equality we used the equivalent definition \eqref{dfn:t_alternative} of $t_n$.


\begin{lem}
Suppose that $t_n$ solves equation (\ref{eqn:t_n}) for $z=u+iv\in\bC^*$, then for all $j=1,\ldots,n$,
$|\beta_j|$ is bounded by $\dfrac{|z-t_n\sigma_n^2|}{v-\Im(t_{n})\sigma_n^2}$.
\label{betabound}
\end{lem}

%

\begin{lem}
Suppose that $t_n$ solves equation (\ref{eqn:t_n})  for $z=u+iv\in\bC^*$, then
$\|\B_n^{-1}\|$ is bounded by $v^{-1}$.
\label{Bbound}
\end{lem}


\begin{lem}
Suppose that $t_n$ solves equation (\ref{eqn:t_n})  for $z=u+iv\in\bC^*$, then
the random variables $\varpi_j$  satisfy
$$\max_{1\le j\le n} E|\varpi_j|^4\le \ \dfrac{C(\log n)^6}{n^2(v-t_{n2}\si_n^2)^4},$$
where $\varpi_j$ can be any of $\eta_j$, $\hat{\eta}_j$, $\gamma_j$ and $\hat{\gamma}_j$ defined in (\ref{etaj}), and $C$ is a constant independent of $n$.
\label{rv4}
\end{lem}

\begin{lem}
Suppose that $t_n$ solves equation (\ref{eqn:t_n}) for $z=u+iv\in\bC^*$, then
the random variables $w_j$ and $\hat{w}_j$ satisfy
\begin{eqnarray*}
&&\max_{1\le j\le n} E\left|w_j-\dfrac{\sigma_n^2}{n}{\rm tr}({\bf R}_n^{-1})\right|^4
\le \ \dfrac{C(\log n)^8}{n^2(v-t_{n2}\sigma_n^2)^4}, \\
&&\max_{1\le j\le n} E\left|\hat{w}_j-\dfrac{\sigma_n^2}{n}{\rm tr}({\bf R}_n^{-1}{\bf B}_n^{-1})\right|^4
\le \ \dfrac{C(\log n)^8}{n^2v^4(v-t_{n2}\sigma_n^2)^4}.
\end{eqnarray*}
\label{ww}
\end{lem}


The proofs of Lemmas \ref{betabound} - \ref{ww} are also given in the Appendix \ref{appendix:pfs}.


\subsubsection{Proof of Proposition \ref{prop:LSD_signal_noise}}\label{ssec:pf_prop_LSD_signal_noise}

\begin{proof}Recall the $\De_j, j=1,2,3$ defined in \eqref{eq:Delta}.
The proof will be completed if we show $\De_j\to 0$ almost surely for all $j=1,2,3$.

By
\eqref{bound_RRj},
\eqref{asm:A_bd}
and Lemma \ref{Bbound}, there exists a constant $C$ such that
\begin{eqnarray}
~~~
\max_{j=1,\ldots,n}  |\rho_j| \le \ \dfrac{C\log(n)}{v-t_{n2}\sigma_n^2}, \q \mbox{and}\q
\max_{j=1,\ldots,n} |\hat{\rho}_j| \le \ \dfrac{C\log(n)}{v(v-t_{n2}\sigma_n^2)}.
\label{rhoj}
\end{eqnarray}
Moreover, by Lemmas \ref{lem:t_properties}, \ref{lem:tlimit} and  the convergence of $\{F^{\S_n}\}$, we have as $p\to\infty$,
\begin{eqnarray}\label{eqn:de_conv}
\de_n=y_n\si_nm_n(z-t_n\si_n^2) ~ \to ~ \de=\de(z)=y\si^2m(z-t\si^2),
\end{eqnarray}
and $\Im(\de)>0$.
In particular, for all $n$ large enough, we have
\begin{eqnarray}
\dfrac{1}{ \ |1+\delta_n| \ } \ \le \ \dfrac{2}{\liminf_n\Im(\de_n)} <\infty.
\label{delta0}
\end{eqnarray}

We now show that $\De_3\to 0$ almost surely.
Using Markov's inequality and H\"{o}lder's inequality, for any $\eps>0$, we have
\begin{eqnarray*}
{\rm P}\(|\De_3|\geq \eps\)&\le& \dfrac{1}{\eps^4} E
\left|\dfrac{1}{p}\sum_{j=1}^n\beta_j \(\hat{w}_j-\dfrac{\sigma_n^2}{n}{\rm tr}({\bf R}_n^{-1}{\bf B}_n^{-1})\)\right|^4\\
&\le& \dfrac{n^3}{p^4\eps^4}\sum_{j=1}^nE|\beta_j|^4
\left|\hat{w}_j-\dfrac{\sigma_n^2}{n}{\rm tr}({\bf R}_n^{-1}{\bf B}_n^{-1})\right|^4\\
&\le&\dfrac{C(\log n)^8}{n^2\eps^4v^4(v-t_{n2}\si_n^2)^8}\cdot|z-t_n\si_n^2|^4,
\end{eqnarray*}
where the last step follows from Lemmas \ref{betabound} and \ref{ww}.
Thus $\De_3\to 0$ almost surely by Lemmas \ref{lem:tlimit}, \ref{lem:t_properties} and the Borel-Cantelli Lemma.

Similarly we can prove that $\De_j\to 0$ almost surely for $j=1,2$ by using Lemmas \ref{betabound}, \ref{Bbound}, \ref{rv4}, \ref{ww} and inequalities (\ref{rhoj}), (\ref{delta0}).

\end{proof}

\subsubsection{Proof of Theorem \ref{thm:LSD_signal_noise}}\label{ssec:thm_LSD_signal_noise}

\begin{proof}
We first show that equation (1.1) in \cite{DS2007a} can be derived from Proposition \ref{prop:LSD_signal_noise}.

For any fixed $z\in\bC^*$, by Proposition \ref{prop:LSD_signal_noise}, Lemmas \ref{lem:tlimit}, \ref{lem:t_properties}, \ref{Bbound},  and the dominated convergence theorem we obtain that
\begin{eqnarray}\label{rem}
m(z-t\si^2) \ = \ \int\dfrac{1}{(1+\de)^{-1}x-z} \  dF^{\cA}(x),
\end{eqnarray}
where $t$ is the unique solution to equation \eqref{eqn:t} and $\de=y\si^2m(z-t\si^2)$.
Moreover, if we let $\ga(z)=z-t(z)\si^2$, then by the definition \eqref{eqn:t} of $t$ and the convergence
\eqref{eqn:de_conv} we have
$$t=y-1+y\ga m(\ga), ~~~~~~~~~~~\de=y\si^2m(\ga),$$
and
\[
z=\ga+t\si^2=\ga+\ga y\si^2 m(\ga)+\si^2(y-1).
\]
Substituting the expressions of $t$, $\de$ and $z$ in terms of $\ga$ into equation (\ref{rem}) yields
\begin{equation}\label{rem_limit}
m(\ga)=\int \dfrac{dF^{\cA}(x)}{\dfrac{x}{1+y\si^2m(\ga)}-\ga(1+y\si^2m(\ga))-\si^2(y-1)},
\end{equation}
where $\ga\in\bC_{\ga}:=\{\ga=z-t(z)\si^2: z\in\bC^*\}$.

Next we show that \eqref{rem_limit} holds for all $\ga\in\bC^+$.
In fact, by Lemma \ref{lem:t_analytic}, $\ga(z)$ is analytic on $\bC^*$. In particular,
for any convergent sequence $\{z^{(m)}\}\subset\bC^*$ such that
$z^{(m)}\to z_\infty\in\bC^*$ as $m\to\infty$,
we have $\ga_m:=\ga(z^{(m)})\to \ga_\infty:=\ga(z_\infty)$,
all in $\bC_{\ga}\subseteq \bC^+$;
moreover,   $\ga_m$ and $\ga_\infty$ all satisfy equation \eqref{rem_limit}.
Noting that equation \eqref{rem_limit} is well-defined for all $\ga\in\bC^+$,  by the analyticity of $m(\ga)$ on $\bC^+$ and the uniqueness theorem for analytic functions, we conclude that equation (\ref{rem_limit}) holds for every $\ga\in\bC^+$, in other words, equation~(1.1) in \cite{DS2007a} holds.

In the following, we will show that equation (\ref{eqn:LSD_signal_to_noisy}) in Theorem \ref{thm:LSD_signal_noise} holds.

For any $z\in\bC^*$, denote $\al(z)=z(1+\de(z))$, where, recall that,  $\de(z)=y\si^2 m(\ga)$ and $\ga=z-t\si^2$.  We further define
\[
d(\ga)=1+y\si^2m(\ga) (=1+\de(z)), ~~~~{\rm and}~~~~ g(\al)=1-y\si^2m_{\cA}(\al).
\]
We will show the following facts:
\begin{compactenum}\setcounter{enumi}{6}
 \item[]
 \begin{compactenum}
  \item\label{F1} $g(\al)=1/d(\ga)$,
  \item\label{F2} $\al=\ga d^2(\ga)+\si^2(y-1)d(\ga)$, or $\ga=\al g^2(\al)-\si^2(y-1)g(\al)$.
\end{compactenum}
\end{compactenum}

In fact, we can rewrite equation \eqref{rem} as
\begin{eqnarray*}
m_{\cA}\(\al\)
=(1+\de)^{-1} m(\ga).
\end{eqnarray*}
Noting that $\de=y\si^2m(\ga),$
we have
\[
y\si^2 m_{\cA}(\al)=\dfrac{\de}{1+\de}, \mbox{ and hence } g(\al)=\frac{1}{1+\delta} = \frac{1}{d(\ga)},
\]
namely, \eqref{F1} holds.
Besides, $y\si^2m_{\cA}(\al)=1-{1}/{(1+\de)}$ implies $\al\in\bC^+$  since
$\de=y\si^2 m(z-t\si^2)\in \bC^+$ by Lemma \ref{lem:t_properties}.

We now show \eqref{F2}.
Let $\beta=t\si^2(1+\de)$. Then
\begin{eqnarray}\label{ga}
\ga=z-t\si^2=\dfrac{\al-\beta}{1+\de}.
\end{eqnarray}
By substituting \eqref{ga} and $\de=y\si^2 m(\ga)$ into equation \eqref{eqn:t}, we obtain
\[
\dfrac{\beta}{\si^2(1+\de)}=y-1+\dfrac{\de(\al-\beta)}{\si^2(1+\de)}.
\]
That is,
\[
\beta=\si^2(y-1)+\dfrac{\de}{1+\de}\al.
\]
Therefore,
\begin{eqnarray*}
\ga&=&\dfrac{\al-\beta}{1+\de}
= \dfrac{\al}{(1+\de)^2}-\dfrac{\si^2(y-1)}{1+\de}\\
&=&\al g^2(\al)- \si^2(y-1)g(\al),
\end{eqnarray*}
namely,  \eqref{F2} holds.

Next, by \eqref{rem_limit} and the definitions of $\al$ and $d(\ga)$ and \eqref{F2}, we have
\begin{eqnarray*}
m(\ga)
&=&d(\ga)\int\dfrac{1}{x-\al}dF^{\cA}(x).
\end{eqnarray*}
Using the facts
\eqref{F1} and \eqref{F2}
we obtain that
\begin{equation}\label{g_a}
\aligned
m_{\cA}(\al)&=\int\dfrac{dF^A(x)}{x-\al}
=\dfrac{1}{d(\ga)}\int\dfrac{1}{\tau-\ga}dF(\tau) \\
&=\int\dfrac{g(\al)}{\tau-\al g^2(\al)+\si^2(y-1)g(\al)} dF(\tau) \\
&=\int\dfrac{1}{\dfrac{\tau}{g(\al)}-\al g(\al)+\si^2(y-1)} dF(\tau).
\endaligned
\end{equation}
By plugging in the expression of $g(\al)$,  we see that for all
$\al=\al(z)=z(1+\de(z))$,
$m_{\cA}(\al)$ satisfies
\[
~~~~~~~ m_{\cA}(\al)=\int\dfrac{dF(\tau)}{\dfrac{\tau}{1-y\si^2 m_{\cA}(\al)}-\al(1-y\si^2 m_{\cA}(\al))+\si^2(y-1)}.
\]
It follows from the uniqueness theorem for analytic functions that the above equation holds for all $\al\in\bC^+$ such that   the integral on the right hand side is well-defined.

It remains to show that  the solution to  equation \eqref{eqn:LSD_signal_to_noisy} is unique  in the set~$D_{\cA}$ defined in \eqref{dfn:mA_domain}.
In fact,
suppose otherwise that $m_1\neq m_2\in D_{\cA}$ both satisfy equation \eqref{eqn:LSD_signal_to_noisy}.  Define for $ j=1,2,$
\begin{eqnarray}\label{ga_j}
\ga_j=\al(1-y\si^2m_j)^2-\si^2(y-1)(1-y\si^2m_j)\in \bC^+.
\end{eqnarray}
By \eqref{eqn:LSD_signal_to_noisy} and \eqref{ga_j}, we have
$
m_j=(1-y\si^2m_j)m(\ga_j).
$
Hence
\begin{eqnarray}\label{m_ga_j}
m(\ga_j)=\dfrac{m_j}{1-y\si^2m_j},  ~~{\rm for}~j=1,2.
\end{eqnarray}
which implies that
\begin{eqnarray}\label{mm_ga_j}
1+y\si^2m(\ga_j)=\dfrac{1}{1-y\si^2m_j}, ~~{\rm for}~j=1,2.
\end{eqnarray}
Using \eqref{ga_j} and \eqref{mm_ga_j} we can rewrite $\al$ as
\begin{eqnarray}\label{al_ga_j}
\al&=&\dfrac{\ga_j}{(1-y\si^2m_j)^2}+\dfrac{\si^2(y-1)}{1-y\si^2m_j}\nonumber\\
&=&\ga_j(1+y\si^2m(\ga_j))^2+\si^2(y-1)(1+y\si^2m(\ga_j)), ~~{\rm for}~j=1,2.
\end{eqnarray}
Observing that the Stieltjes transforms $m(\ga_1)$ and $m(\ga_2)$ are uniquely determined by equation \eqref{rem_limit} at points $\ga_1$ and $\ga_2$ respectively, together with \eqref{al_ga_j}, we obtain
\begin{eqnarray*}
m(\ga_j)&=&\int\dfrac{dF^{\cA}(x)}{\dfrac{x}{1+y\si^2m(\ga_j)}-\ga_j(1+y\si^2m(\ga_j))-\si^2(y-1)}\\
&=&(1+y\si^2m(\ga_j)) \cdot m_{\cA}(\al), ~~{\rm for}~j=1,2.
\end{eqnarray*}
Therefore
\[
\dfrac{m(\ga_1)}{1+y\si^2m(\ga_1)}=\dfrac{m(\ga_2)}{1+y\si^2m(\ga_2)},
\]
which implies that $m(\ga_1)=m(\ga_2)$. It then follows from \eqref{m_ga_j} that $m_1=m_2$, a contradiction.
\end{proof}


\section{Conclusion\label{sec:conclusion} }
Motivated by the inference about the spectra of the \ICV  matrix based on high-frequency noisy data,
\begin{compactitem}
\item we establish an asymptotic relationship that describes how the spectral distribution of (true) sample  covariance matrices depends on that of sample covariance matrices constructed from noisy observations;
\item using further a (generalized) connection between the spectral distribution of true sample covariance matrices and that of the population covariance matrix, we propose a two-step procedure to estimate the spectral distribution of ICV for a class  of diffusion processes;
\item  we further develop an alternative estimator which possesses two desirable properties: it eliminates the impact of microstructure noise, and its limiting spectral distribution depends only on that of the ICV through the standard Mar\v{c}enko-Pastur equation;
\item numerical studies demonstrate that our proposed methods can be used to estimate the spectra of the underlying covariance matrix based on noisy observations.
\end{compactitem}


%



\appendix

\renewcommand{\baselinestretch}{1.2}
\setcounter{equation}{0}
\renewcommand{\theequation}{\thesection.\arabic{equation}}

\section{Proofs}\label{appendix:pfs}
\subsection{Proof of Lemmas \ref{lem:tn_exist} -- \ref{ww}}

\begin{proof}[Proof of Lemma \ref{lem:tn_exist}]
Rewrite equation (\ref{eqn:t_n}) as
\begin{equation}\label{re_t}
\aligned
t_n+1 &=
y_n+y_n\int \ \dfrac{z-t_n\sigma_n^2}{x-z+t_n\sigma_n^2} \ dF^{S_n}(x) \\
&= \ y_n\int \ \dfrac{x}{x-z+t_n\sigma_n^2} \ dF^{S_n}(x).
\endaligned
\end{equation}

Firstly, under the assumptions of Theorem \ref{thm:LSD_signal_noise}, by Theorem 1.1 in \cite{Bai2012_noise}, if we let $[a_n,b_n]$ be an interval containing the support of~$F^{\S_n}$, then we may assume that  for all large $n$, $b_n\leq \wt{b}:=b+1$.
Let $\wt{\si}=\si+1$, $\wt{y}=y+1$  and $K_1=2\wt{\si}\sqrt{\wt{y}\wt{b}}$. Since $\si_n\to\si$ and $y_n\to y$, we have for all large $n$ and for all $t\in\sD$,
\begin{equation}\label{eqn:v_t_lower_bd}
\si_n< \wt{\si}, ~ y_n< \wt{y},~~~ {\rm and}~v-t_2\si_n^2\ge v-t_2\wt{\si}^2\geq v/2>0.
\end{equation}

Define
\[
G(t) \ = \ y_n \int \ \dfrac{x}{x-z+t\sigma_n^2} \ dF^{S_n}(x)-1, ~~~~~ {\rm for ~all}~~ t\in \mathscr{D}.
\]
We will apply the Banach fixed point theorem to show that  for all $n$ large enough, there exists a unique point $t^*\in \mathscr{D}$ such that $G(t^*)=t^*$. The desired conclusion then follows.

\ul{Step (i)}: we prove that the mapping $G$ is defined from $\mathscr{D}$ to $\mathscr{D}$. From the definition of $G(t)$ and that $t\in\sD$,   we have
\begin{eqnarray*}
\Im(G(t))&=&y_n\int_{a_n}^{b_n}\frac{x(v-t_2\sigma_n^2)}{(x-u+t_1\sigma_n^2)^2+(v-t_2\sigma_n^2)^2} \ dF^{S_n}(x)\\
&=&\dfrac{y_n}{v-t_2\sigma_n^2}\int_{a_n}^{b_n}\dfrac{x}{1+\left(\frac{x-u+t_1\sigma_n^2}{v-t_2\sigma_n^2}\right)^2} \ dF^{S_n}(x),
\end{eqnarray*}
and hence for all $n$ large enough,
\begin{eqnarray*}
0<\Im(G(t))
< \dfrac{\wt{y}\wt{b}}{v-t_2\wt{\si}^2}
\le \dfrac{v}{2\wt{\si}^2},
\end{eqnarray*}
where the last inequality follows from the fact that for any  $z\in\bC_1$,
\[
\dfrac{\wt{y}\wt{b}}{v-t_2\wt{\si}^2}-\dfrac{v}{2\wt{\si}^2}
\le \dfrac{2\wt{y}\wt{b}}{v}-\dfrac{v}{2\wt{\si}^2}
= \dfrac{4\wt{\si}^2\wt{y}\wt{b}-v^2}{2\wt{\si}^2v}
\le 0.
\]

\ul{Step (ii)}: we shall show that $G: \mathscr{D} \to \mathscr{D}$ is a contraction mapping. In fact, for any two points $t$, $t^{\prime} \ \in \mathscr{D}$,
\begin{eqnarray*}
G(t)-G(t^{\prime})&=& y_n\int_{a_n}^{b_n} \
\left(\dfrac{x}{x-z+t\sigma_n^2}-\dfrac{x}{x-z+t^{\prime}\sigma_n^2}\right)
\ dF^{S_n}(x)\\
&=& (t-t^{\prime}) \ y_n\sigma_n^2 \int_{a_n}^{b_n} \
\dfrac{-x}{(x-z+t\sigma_n^2)(x-z+t^{\prime}\sigma_n^2)}
\ dF^{S_n}(x)\\
&:=& (t-t^{\prime}) \ q(t,t^{\prime}).
\end{eqnarray*}

Using Cauchy-Schwartz inequality we get that almost surely for all $n$ large enough, for all $t,t^{\prime}\in\mathscr{D}$,
\[
\aligned
&|q(t,t^{\prime})|\\
\le& \left(\int_{a_n}^{b_n} \ \dfrac{\sigma_n^2 y_n x}{|x-z+t\sigma_n^2|^2} \ dF^{S_n}(x)\right)^{1/2}
\left(\int_{a_n}^{b_n} \ \dfrac{\sigma_n^2 y_n x}{|x-z+t^{\prime}\sigma_n^2|^2} \ dF^{S_n}(x)\right)^{1/2}\\
\le& \left(\dfrac{\sigma_n^2y_nb_n}{(v-\Im(t)\sigma_n^2)^2}\right)^{1/2}
\left(\dfrac{\sigma_n^2y_nb_n}{(v-\Im(t^{\prime})\sigma_n^2)^2}\right)^{1/2}\\
<&\left(\dfrac{\wt{\si}^2\wt{y}\wt{b}}{(v-\Im(t)\wt{\si}^2)^2} \right)^{1/2} \left(\dfrac{\wt{\si}^2\wt{y}\wt{b}}{(v-\Im(t^{\prime})\wt{\si}^2)^2} \right)^{1/2}\\
\le&\left(\dfrac{\wt{\si}^2\wt{y}\wt{b}}{v^2/4} \right)^{1/2} \left(\dfrac{\wt{\si}^2\wt{y}\wt{b}}{v^2/4} \right)^{1/2},
\endaligned
\]
which is strictly smaller than 1 when $z\in\bC_1$.
Therefore the mapping $G$ is contractive  in $\mathscr{D}$, and the Banach fixed point theorem guarantees the existence of a unique solution to equation (\ref{eqn:t_n}).
\end{proof}


\begin{proof}[Proof of Lemma \ref{lem:t_properties}]
Taking imaginary parts on both sides of equation (\ref{eqn:t}) yields
\begin{eqnarray}\label{expression_t2}
t_2= y \int_a^b \dfrac{x(v-t_2\si^2)}{|x-z+t\si^2|^2} \ dF(x).
\end{eqnarray}
It is then straightforward to verify that
$t_2>0$ and $v-t_2\si^2>0$.
Furthermore, since
\begin{equation}\label{sim_t2}
\aligned
t_2=&\dfrac{y}{v-t_2\si^2} \ \int_a^b \dfrac{x}{1+\(\dfrac{x-u+t_1\si^2}{v-t_2\si^2}\)^2} \ dF(x)\\
\le&\dfrac{yb}{v-t_2\si^2},
\endaligned
\end{equation}
when $v\geq 2\si\sqrt{yb}$, we have
\begin{equation}\label{eq:t_2_possibilities}
{\rm either} ~~~~~
t_2 \ \geq  \ \dfrac{ \ v \ + \ \sqrt{v^2-4\sigma^2 yb}}{2\sigma^2} ~~~~~ \textrm{or} ~~~~~
t_2 \ \le \ \dfrac{ \ v \ - \ \sqrt{v^2-4\sigma^2 yb}}{2\sigma^2}.
\end{equation}

Denote $w=u-t_1\si^2$ and $\th=v-t_2\si^2$.
By \eqref{sim_t2}, if $F$ admits a bounded density $f$ and possibly a point mass at 0, then
\begin{eqnarray*}
t_2&=&\dfrac{y}{\theta} \ \int_a^b \ \dfrac{x}{ 1+\(\dfrac{x-w}{\theta}\)^2 } \  f(x) \ dx\\
&=& y \ \int_{\frac{a-w}{\theta}}^{\frac{b-w}{\theta}} \ \dfrac{ \ w+\theta l \ }
{ \ 1+l^2 \ } \ f(w+\theta l) \ dl.
\end{eqnarray*}
Since $f(w+\theta l)$ is bounded  and $x=w+\theta l \ \in(a,b)$ when $l\in (\frac{a-w}{\th},\frac{b-w}{\th})$,  there exists a constant $C$ such that
\[
t_2 \ \le \ C\int_{\frac{a-w}{\theta}}^{\frac{b-w}{\theta}} \
\dfrac{1}{1+l^2} \ dl \  \le \ C\int_{-\infty}^{+\infty} \ \dfrac{dl}{1+l^2} \ = \ C \ \pi.
\]
This, combined with \eqref{eq:t_2_possibilities}, implies that
\begin{eqnarray}\label{eqn:t2_upper_bd}
t_2 \ \le \ \dfrac{v \ - \ \sqrt{v^2-4\sigma^2 yb}}{2\sigma^2},\q\mbox{for all } v \mbox{ large enough}.
\end{eqnarray}
In particular,  uniformly in $u$,
\begin{equation}\label{eqn:t2_to_0_as_v_infty}
t_2\to 0 \mbox{ and } v-t_2\sigma^2\to\infty,\q \mbox{ as }v\to\infty.
\end{equation}

Moreover, from (\ref{eqn:t}) we get
\[
t+1 \ = \ y \ + \ y\int \ \dfrac{z-t\sigma^2}{ \ x-z+t\sigma^2 \ } \ dF(x) \
=y\int\dfrac{x}{ \ x-z+t\sigma^2 \ }  \  dF(x).
\]
Thus as $v\to\infty$,
\begin{eqnarray*}
|t_1+1| \le |t+1|
\le  y\int_a^b\dfrac{x}{ \ \Im(x-z+t\sigma^2) \ } \ dF(x)
\le  \dfrac{C}{v-t_2\sigma^2} \to ~ 0,
\end{eqnarray*}
also uniformly in $u$.
\end{proof}


\begin{proof}[Proof of Lemma \ref{lem:t_exist}]
Firstly, by the same proof as for Lemma \ref{lem:tn_exist}, one can show that for all $z=u+iv$ with $v\geq K_1$,  equation \eqref{eqn:t} admits a unique solution in $\sD$ defined in \eqref{dfn:D}. Moreover, by Lemma \ref{lem:t_properties}, if $t=t_1+it_2$ solves \eqref{eqn:t}, then $t_2>0$; furthermore, we can find a constant $K_2$ such that
if $t$ solves~\eqref{eqn:t} for~$z$ with $v(=\Im(z))\geq K_2,$ then we must have $t_2\le{v}/{(2\wt{\si}^2)}$. The latter two properties imply that for all  $z$ with $v\geq K_2,$ the solution to \eqref{eqn:t} must lie in $\sD$. Redefining $K_2=\max(K_1,K_2)$ if necessary, we see that for all $z\in\bC_2$,  \eqref{eqn:t} admits a unique solution.
\end{proof}

\begin{proof}[Proof of Lemma \ref{lem:t_analytic}]
Define a function $G$ as
\[
G(z,t)=t-(y-1)-y(z-t\si^2)m(z-t\si^2), \ (z,t)\in \bC^+\times\bC^+ \mbox{ with } \Im(z-t\si^2)>0.
\]
That $t(z)$ solves \eqref{eqn:t} is equivalent to $G(z,t(z))=0$.
Write $z=u+iv$ and $t=t_1+it_2$.
By taking the partial derivative with respect to $t$ we get
\[
\dfrac{\partial G}{\partial t}=1+y\si^2
\int\dfrac{x}{\(x-(z-t\si^2)\)^2} ~ dF(x).
\]
Note that
\[
\left|\int\dfrac{x}{\(x-(z-t\si^2)\)^2} ~ dF(x)\right| ~
\le ~ \dfrac{b}{(v-t_2\si^2)^2},
\]
which, by \eqref{eqn:t2_to_0_as_v_infty}, goes to zero as $v\to\infty$.
Thus there exists a constant $K_3>0$ such that for all $z\in \bC_3$, ${\partial G}/{\partial t}(z,t(z))\neq 0$. It follows from the implicit function theorem and Lemma \ref{lem:t_properties} that $t=t(z)$ is  analytic on $\bC_3$.
\end{proof}

\begin{proof}[Proof of Lemma \ref{lem:tlimit}]
Write $z=u+iv$ and $t_n=t_{n1}+it_{n2}$.
Similar to the proof of Lemma \ref{lem:t_properties}, taking imaginary parts on both sides of equation (\ref{eqn:t_n}), one can easily show that
$t_{n2}>0$ and $v-t_{n2}\si_n^2>0$.

Next we show that $\{t_n\}$ is tight, in other words, for any $\eps>0$, there exists $C>0$, such that for all $n$ large enough, $P\left(|t_n|>C\right)<\eps$.
Since $0<t_{n2}<v/\sigma_n^2$, it suffices to show that $\{|t_{n1}|\}$ is tight.

Let $\ul{\S}_n= \dfrac{1}{n}(\A_n+\si_n\pvep_n)^T(\A_n+\si_n\pvep_n)$, and let $\ul{m}_n(z)$ be the Stieltjes transform of the ESD $F^{\ul{\S}_n}$. The spectra of $\S_n$ and $\ul{\S}_n$ differ by $|p-n|$ number of zero eigenvalues,  hence
$
F^{\ul{\S}_n}=(1-y_n) I_{[0,\infty)} +y_n F^{\S_n},
$
and
\begin{eqnarray}
\ul{m}_n(z)=-\dfrac{1-y_n}{z}+y_n m_n(z).
\label{mmn}
\end{eqnarray}
Thus equation (\ref{eqn:t_n}) can also be expressed as
\begin{eqnarray*}
t_n&=&y_n-1+y_n(z-t_n\si_n^2)m_n(z-t_n\si_n^2) \nonumber\\
&=&(z-t_n\si_n^2)\ul{m}_n(z-t_n\si_n^2).
\end{eqnarray*}
Taking real parts on both sides yields
\[
\Re(t_{n})=\int\dfrac{x(u-\Re(t_{n})\si_n^2)-|z-t_n\si_n^2|^2}{|x-z+t_n\si_n^2|^2} \ dF^{\ul{\S}_n}(x).
\]
Solving for $\Re(t_{n})$ yields
\begin{equation}\label{re_tn}
\Re(t_{n}) = \dfrac{\displaystyle\int \dfrac{xu-|z-t_{n}\si_{n}^2|^2}{|x-z+t_{n}\si_{n}^2|^2}\ dF^{\ul{\S}_{n}}(x)}
{1+\si_{n}^2\displaystyle\int \dfrac{x}{|x-z+t_{n}\si_{n}^2|^2}\ dF^{\ul{\S}_{n}}(x)}
\end{equation}

Now suppose that $\{t_{n1}=\Re(t_n)\}$ is not tight, then with positive probability, there exists a subsequence $\{n_k\}$ such that $|\Re(t_{n_k})|\to\infty$.
By \eqref{re_tn}, we have
\begin{eqnarray*}
|\Re(t_{n_k})|
&\le& \int_{a_{n_k}}^{b_{n_k}} \dfrac{x|u|+|z-t_{n_k}\si_{n_k}^2|^2}{|x-z+t_{n_k}\si_{n_k}^2|^2} ~ dF^{\ul{\S}_{n_k}}(x).
\end{eqnarray*}
However, as $k$ goes to infinity, if  $|\Re(t_{n_k})|\to\infty$, since $\{F^{\ul{\S}_{n_k}}\}$ is tight and $\si_{n_k}\to\si>0$, one gets that the RHS goes to 1. This contradicts the supposition that $|\Re(t_{n_k})|\to\infty$.

Next, for any convergent subsequence $\{t_{n_k}\}$ in  set $\sD$, by \eqref{eqn:v_t_lower_bd}, for all~$n_k$ large enough, we have $v-\Im(t_{n_k})\si_{n_k}^2\geq v/2$. We can then apply the dominated convergence theorem  to conclude that the limit point of $\{t_{n_k}\}$ must satisfy equation \eqref{eqn:t}.
By Lemma \ref{lem:t_exist}, the solution is unique, hence the whole sequence $\{t_n\}$ converges to the unique solution to equation \eqref{eqn:t}.
\end{proof}

\begin{proof}[Proof of Lemma \ref{betabound}]
Write $t_n=t_{n1}+it_{n2}$. Note that
\begin{eqnarray*}
&&\Im\left\{(z-t_n\sigma_n^2){\pmb\xi}_j^T{\bf R}_{nj}^{-1}{\pmb\xi}_j\right\}\\
&&=\Im\left\{{\pmb\xi}_j^T\left(\dfrac{1}{z-t_n\sigma_n^2}{\bf S}_{nj}-\I\right)^{-1}{\pmb\xi}_j\right\}\\
&&=\dfrac{1}{2i}{\pmb\xi}_j^T\left[\left(\dfrac{1}{z-t_n\sigma_n^2}{\bf S}_{nj}-\I\right)^{-1}
-\left(\dfrac{1}{\overline{z-t_n\sigma_n^2}}{\bf S}_{nj}-\I\right)^{-1}\right]{\pmb\xi}_j\\
&&=\dfrac{v-t_{n2}\sigma_n^2}{|z-t_n\sigma_n^2|^2} \ {\pmb\xi}_j^T\left(\dfrac{1}{z-t_n\sigma_n^2}{\bf S}_{nj}-\I\right)^{-1}
{\bf S}_{nj}\left(\dfrac{1}{\overline{z-t_n\sigma_n^2}}{\bf S}_{nj}-\I\right)^{-1}{\pmb\xi}_j \\
&&\geq 0,
\end{eqnarray*}
where the last inequality is due to Lemma \ref{lem:tlimit}. Therefore,
\begin{eqnarray*}
|\beta_j|&=& \dfrac{|z-t_n\sigma_n^2|}{ \ |(z-t_n\sigma_n^2)(1+{\pmb\xi}_j^T{\bf R}_{nj}^{-1}{\pmb\xi}_j)| \ } \\
&\le& \dfrac{|z-t_n\sigma_n^2|}{ \ |\Im\{(z-t_n\sigma_n^2)(1+{\pmb\xi}_j^T{\bf R}_{nj}^{-1}{\pmb\xi}_j)\}| \ }  \\
&\le&   \dfrac{ \ |z-t_n\sigma_n^2| \ }{ \ v-t_{n2}\sigma_n^2 \ }.
\end{eqnarray*}

\end{proof}


\begin{proof}[Proof of Lemma  \ref{Bbound}]
Any eigenvalue of $\B_n=\dfrac{1}{n(1+\delta_n)}{\bf A}_n{\bf A}_n^T-z\I$
can be expressed as $\la^B=\dfrac{1}{1+\de_n}\la-z$, where $\la$ is an eigenvalue of $\dfrac{1}{n}\A_n\A_n^T$. We have
\[
|\lambda^B|\geq |\Im(\lambda^B)|=\left|\dfrac{\Im(\delta_n)}{|1+\delta_n|^2}\lambda+v\right|\geq v,
\]
where the last step follows from the fact that
$
\Im(\delta_n)=y_n\sigma_n^2\Im(m_n(z-t_n\sigma_n^2)) \ >0,
$
thanks to Lemma \ref{lem:tlimit}.
\end{proof}

\begin{proof}[Proof of Lemma \ref{rv4}]
We shall only establish the inequality for $\eta_j(=\ga_j)$; the other two variables $\hat{\eta}_j$ and $\hat{\gamma}_j$ can be handled in a similar way by using Lemma \ref{Bbound}.

Since for any Hermitian matrix $\A$ and $z\in\bC^+$, $\|(\A-z\I)^{-1}\|\le 1/\Im(z)$, we have by Lemma \ref{lem:tlimit}  that
\begin{eqnarray}\label{bound_RRj}
~~~~\|{\bf R}_n^{-1}\|\le \dfrac{1}{(v-t_{n2}\sigma_n^2)}, ~~~{\rm and}~~~
\max_{1\le j\le n}\|{\bf R}_{nj}^{-1}\|\le \dfrac{1}{(v-t_{n2}\sigma_n^2)}.
\end{eqnarray}

Recall that ${\bf b}_j=\sigma_n \pep_j$, and $\pep_j$ satisfies   $E(\pep_j\pep_j^T)=\I$.
The strengthened assumption \eqref{asm:A_bd} implies that $|{\bf a}_j|\le C\sqrt{n\log n}$.
Note also that $\pep_j$ is independent of ${\bf R}_{nj}^{-1}$ and ${\bf a}_j$. Moreover,  using Lemma \ref{xtrx} in  Appendix \ref{appendix:lemmas},
assumption \eqref{asm:eps_bd} and \eqref{bound_RRj}, we get
\begin{eqnarray*}
E|\eta_j|^4&=&\dfrac{1}{n^4}E|{\bf a}_j^T{\bf R}_{nj}^{-1}{\bf b}_j|^4
=\dfrac{\sigma_n^4}{n^4}E|{\bf a}_j^T{\bf R}_{nj}^{-1}{\bf \pep}_j|^4\\
&=&\dfrac{\sigma_n^4}{n^4}E\left({\bf \pep}_j^T\bar{{\bf R}}_{nj}^{-1}{\bf a}_j{\bf a}_j^T{\bf R}_{nj}^{-1}{\bf \pep}_j\right)^2\\
&\le&\dfrac{2\sigma_n^4}{n^4}\left(
E|{\bf \pep}_j^T\bar{{\bf R}}_{nj}^{-1}{\bf a}_j{\bf a}_j^T{\bf R}_{nj}^{-1}{\bf \pep}_j-{\bf a}_j^T{\bf R}_{nj}^{-1}\bar{{\bf R}}_{nj}^{-1}{\bf a}_j|^2
+E({\bf a}_j^T{\bf R}_{nj}^{-1}\bar{{\bf R}}_{nj}^{-1}{\bf a}_j)^2
\right)\\
&\le&\dfrac{C}{n^4}E|\ep_{11}|^4\times E\left({\bf a}_j^T{\bf R}_{nj}^{-1}\bar{{\bf R}}_{nj}^{-1}{\bf a}_j\right)^2\\
&\le& \dfrac{C(\log n)^6}{n^2(v-t_{n2}\si_n^2)^4}.
\end{eqnarray*}
\end{proof}

\begin{proof}[Proof of Lemma \ref{ww}]
Using \eqref{asm:eps_bd}, \eqref{bound_RRj}, Lemmas \ref{Bbound} and \ref{xtrx}, and Lemma 2.6 in \cite{SB95}, we obtain
\begin{eqnarray*}
&&E\left|w_j-\dfrac{\sigma_n^2}{n}{\rm tr}({\bf R}_n^{-1})\right|^4\\
&\le& C\(E\left|\dfrac{\sigma_n^2}{n}\pep_j^T{\bf R}_{nj}^{-1}\pep_j-\dfrac{\sigma_n^2}{n}{\rm tr}({\bf R}_{nj}^{-1})\right|^4
+E\left|\dfrac{\sigma_n^2}{n}{\rm tr}({\bf R}_{nj}^{-1}-{\bf R}_n^{-1})\right|^4\)\\
&\le&\dfrac{C}{n^4}\left|E\((\log n)^4{\rm tr}\({\bf R}_{nj}^{-1}\bar{{\bf R}}_{nj}^{-1}\)\)^2
+(\log n)^8 E{\rm tr}\({\bf R}_{nj}^{-1}\bar{{\bf R}}_{nj}^{-1}\)^2\right|
+\dfrac{C}{n^4(v-t_{n2}\si_n^2)^4}\\
&\le&\dfrac{C(\log n)^8}{n^2(v-t_{n2}\sigma_n^2)^4}.
\end{eqnarray*}
The result for $\hat{w}_j$ can be proved similarly.
\end{proof}


%


\subsection{Proof of Theorem \ref{thm:main}}

\begin{proof}[Proof of  Theorem \ref{thm:main}]
The convergence of $F^{\ICV}$ follows easily from Assumption  \eqref{asm:Sigma_conv} and the fact that
\[
F^{\ICV}(x) \ = \ F^{\bpSi} \left(\frac{x}{\int_0^1\ga_t^2\,dt}\right) ~~~~~~~ {\rm for~all~} x\geq 0.
\]

Next, by Theorem 3.2 in \cite{DS2007b}, the assumption that $F$ has a bounded support implies that $H$ has a bounded support as well. Thus Assumption (A.iii$'$) in \cite{ZL11} that $H$ has a finite second moment is satisfied.

We proceed to show the convergence of $\cA_m$.
As discussed in Subsection~\ref{sssec:inf_PAV}, if
the diffusion process $({\bf X}_t)$ belongs to Class~$\mathcal{C}$, the drift process ${\pmb \mu}_t\equiv 0$, and $(\gamma_t)$ is independent of $({\bf W}_t)$, then conditional on $\{\gamma_t\}$, we have
\begin{eqnarray}
\Delta\overline{{\bf X}}_{2i} \ \stackrel{d}{=} \
\sqrt{w_i}
\ \breve{{\pmb\Sigma}}^{1/2} \ {\bf Z}_i,
\label{rx}
\end{eqnarray}
where $w_i$ is as in \eqref{wie}
and is independent of ${\bf Z}_i$, and ${\bf Z}_i=(Z_i^{1},\ldots,Z_i^{p})^T$ consists of independent standard normals.
Hence, $\cA_m$ has the same distribution as  $\wt{\cA}_m$ defined as
\begin{eqnarray}
\wt{\cA}_m &=& \sum_{i=1}^m w_i \ \breve{\pSi}^{1/2} \Z_i \Z_i^T \breve{\pSi}^{1/2}.
\label{Sm}
\end{eqnarray}

{\bf Claim 1}. Without loss of generality, we can assume that the drift process ${\pmb\mu}_t\equiv 0$ and $(\gamma_t)$ is independent of $({\bf W}_t)$.

In fact, firstly whether the drift term $({\pmb\mu}_t)$ vanishes or not does not affect the LSD of $\cA_m$.
To see this, note that $\De\ol{\X}_{2i}=\wt{\V}_i+\wt{\Z}_i$, where
\begin{eqnarray}\label{Vi}
\wt{\V}_i= \sum_{|j|<k}\(1-\dfrac{|j|}{k}\)\int_{((2i-1)k+j-1)/n}^{((2i-1)k+j)/n} {\pmb\mu}_t\, dt,
\end{eqnarray}
and
\begin{eqnarray}\label{Zi}
\wt{\Z}_i=\pLa \cdot \sum_{|j|<k}\(1-\dfrac{|j|}{k}\) \int_{((2i-1)k+j-1)/n}^{((2i-1)k+j)/n} \ga_t\, d\W_t.
\end{eqnarray}
Since all the entries of $\wt{\V}_i$ are of order $O(k/n)=o(1/\sqrt{p})$, by  Lemma \ref{lemma1} in  Appendix \ref{appendix:lemmas},
$\cA_m$ and
$\sum_{i=1}^m \wt{\Z}_i (\wt{\Z}_i)^T$
have the same LSD.

Next, by the same argument as in the beginning of Proof of Theorem 1 in \cite{ZL11}, we can assume without loss of generality that $(\ga_t)$ is independent of $(\W_t)$.
It follows that $\cA_m$ and $\wt{\cA}_m$ have the same LSD.

{\bf Claim 2}.  $\max_{i,n} (mw_i)$ is bounded, and  there exists a piecewise continuous process $(w_s)$ with finitely many jumps such that
\begin{eqnarray}\label{eqn:wt_conv}
\lim_{n\to\infty}\sum_{i=1}^m\int_{((2i-2)k)/n}^{2ik/n} \ |mw_i-w_s| \ ds=0.
\end{eqnarray}

In fact, using the boundedness of $(\ga_t)$ assumed in \eqref{asm:gamma_conv}   and that $k=\lfloor\theta\sqrt{n}\rfloor$, one can easily show that  $\max_{i,n} (mw_i)$ is bounded.

Next we show that \eqref{eqn:wt_conv} is satisfied for
$w_s=(\ga_s^*)^2/3$.
Define
\begin{eqnarray*}
w_i^*&=&
\sum_{|j|<k} \(1-\dfrac{|j|}{k}\)^2 ~
 \int_{((2i-1)k+j-1)/n}^{((2i-1)k+j)/n} (\ga_t^*)^2 \, dt.
\end{eqnarray*}

Suppose that $(\ga_t^*)$ has $J$ jumps for $J\geq 1$. For each $j=1,\ldots,J$, there exists an $\ell_j$ such that the $j$th jump falls in the interval $[{(2\ell_j-2)k}/{n}, \ {(2\ell_j k)}/{n})$. Then
\begin{eqnarray*}
&& \sum_{i=1}^m\int_{((2i-2)k)/n}^{2ik/n} \ |mw_i-w_s| \ ds\\
&=&\sum_{\ell_j\in\{\ell_1,\ldots,\ell_J\}}
\int_{((2\ell_j-2)k)/n}^{2\ell_jk/n} \ |mw_{\ell_j}-w_s| \ ds\\
&&+\sum_{i\not\in\{\ell_1,\ldots,\ell_J\}}
\int_{((2i-2)k)/n}^{2ik/n} \ |mw_i-w_s| \ ds\\
&:=&\De_1+\De_2.
\end{eqnarray*}

Since  $(mw_{\ell_j})$ and $|\ga_s^*|$ are both bounded, for any $\varepsilon>0$ and for $n$ large enough, we have
\[
|\De_1|\le \dfrac{2k}{n}\cdot JC < \varepsilon.
\]

For the second term $\De_2$, since $(\ga_t^*)$ is continuous in $[{(2i-2)k}/{n},{(2ik)/{n}}]$ when $i\not\in\{\ell_1,\ldots,\ell_J\}$, and by (A.vi), $(\ga_t)$ uniformly converges to $(\ga_t^*)$,  for any $\varepsilon>0$ and for $n,p$ large enough, we have
\[
|\ga_t^*-\ga_{(2i-2)k/n}^*|<\varepsilon \mbox{ for all } t\in \left[\frac{(2i-2)k}{n},{\frac{2ik}{n}}\right], \mbox{ and } |\ga_t-\ga_t^*|<\varepsilon\mbox{ for all } t.
\]
Moreover, since $|\ga_t|\le C_2$, for all large $n$ we have
\[\aligned
&|\De_2|\\
\le&
\sum_i\int_{(2i-2)k/n}^{2ik/n} |mw_i-mw_i^*|ds \\
&+\sum_i\int_{(2i-2)k/n}^{2ik/n}
\left|mw_i^*-  \frac{m}{n}  \(\ga_{(2i-2)k/n}^*\)^2 \cdot
 \sum_{|j|<k} \(1-\dfrac{|j|}{k} \)^2
\right| ds\\
&+\sum_i\int_{(2i-2)k/n}^{2ik/n}
\left| \frac{m}{n}   \(\ga_{(2i-2)k/n}^*\)^2\cdot
 \sum_{|j|<k} \(1-\dfrac{|j|}{k} \)^2
-\dfrac{(\ga_s^*)^2}{3}
\right| ds\\
\le&
m^2\cdot\dfrac{2k}{n}\cdot \dfrac{1}{k^2}
\(2k(k+1)(2k+1)/6-k^2\)(2C_2\varepsilon)\\
&+
m^2\cdot\dfrac{2k}{n}\cdot \dfrac{1}{k^2}
\(2k(k+1)(2k+1)/6-k^2\)(2C_2\varepsilon)\\
&+\frac{m}{nk^2}
\(2k(k+1)(2k+1)/6-k^2\)\cdot
\sum_i\int_{(2i-2)k/n}^{2ik/n}
\left|\(\ga_{(2i-2)k/n}^*\)^2-(\ga_s^*)^2\right|^2 ds\\
&+{C_2^2}\cdot m\cdot \dfrac{2k}{n}\cdot
\(
\dfrac{m}{n k^2}
\(2k(k+1)(2k+1)/6-k^2\)-
\dfrac{1}{3}
\)\\
\le& C\varepsilon.
\endaligned
\]
This completes the proof of \eqref{eqn:wt_conv}.

Since $F^{\breve{\pSi}}\to\breve{H}$ and $\breve{H}(x/\zeta)=H(x)$ for $x\geq 0$, using Claim 2 and applying Theorem 1 in \cite{ZL11} we conclude that
the ESD of $\cA_m$ converges to $F^{\cA}$ whose Stieltjes transform satisfies
\begin{eqnarray}
m_{\cA}(z)&=&-\dfrac{1}{z}\int\dfrac{1}{\tau M(z)+1}d\breve{H}(\tau) \nonumber\\
&=&-\dfrac{1}{z}\int\dfrac{\zeta}{\tau M(z)+\zeta}dH(\tau),
\label{m_A}
\end{eqnarray}
where $M(z)$, together with another function $\wt{m}(z)$, uniquely solve the following equations in $\bC^+\times \bC^+$
\begin{equation}\label{mpw}
\left\{
\begin{array}{lll}
M(z) &=& -\dfrac{1}{z} \displaystyle\int_0^1 \dfrac{w_s}{1+y\wt{m}(z) w_s}ds,  \\
\wt{m}(z) &=& -\dfrac{1}{z} \displaystyle\int \dfrac{\tau}{\tau M(z)+1} d\breve{H}(\tau)
{=-\dfrac{1}{z} \displaystyle\int \dfrac{\tau}{\tau M(z)+\zeta} dH(\tau)}.
\end{array}
\right.
\end{equation}

\end{proof}


\subsection{Proof of  Theorem \ref{thm:B_n}}

Note that the convergence of the ESD of $\ICV$ has been proved in Theorem \ref{thm:main}. The rest of Theorem \ref{thm:B_n} is a direct consequence of the following two convergence results.

\begin{lem}\label{pthm3_a}
Under the assumptions of Theorem \ref{thm:B_n}, we have
\[
\lim_{p\to\infty} 3 \dfrac{\sum_{i=1}^m |\De\ol{\Y}_{2i}|^2}{p}
=\zeta, \q \mbox{almost surely}.
\]
\end{lem}

\begin{prop}\label{pthm3}
Under the assumptions of Theorem \ref{thm:B_n},
 $F^{\widetilde{\pSi}}$ converge almost surely, and the limit $\widetilde{F}$
is determined by $\breve{H}$ in that its Stieltjes transform $m_{\widetilde{F}}(z)$ satisfies the following equation
\begin{equation}\label{eqn:B_n}
m_{\widetilde{F}}(z) \ = \ \int_{\tau\in\mathbb{R}}\dfrac{1}{\tau\(1-y(1+zm_{\widetilde{F}}(z))\)-z} \ d\breve{H}(\tau), ~~~ {\rm for~ all}~ z\in\mathbb{C}^+.
\end{equation}
\end{prop}

We will prove Proposition \ref{pthm3} first, and then give the proof of Lemma~\ref{pthm3_a} afterwards.

\begin{proof}[Proof of Proposition \ref{pthm3}]
We now show the convergence of $F^{\tSi}$.
The main reason that we choose~$k$ in such a way that $k/\sqrt{n}\to\infty$ is to make the noise term negligible. To be more specific, by choosing $k=\lfloor\th n^{\alpha}\rfloor$ for
some $\alpha\in [(3+\ell)/(2\ell+2) ,1)$ where $\ell$ is the integer as in Assumption \eqref{asm:eps_general},
 we shall show that
\[
\tpSi=y_m\sum_{i=1}^m\dfrac{\De\ol{\Y}_{2i}(\De\ol{\Y}_{2i})^T}{|\De\ol{\Y}_{2i}|^2}
~~~{\rm and }~~~
\widetilde{\tpSi}:=y_m\sum_{i=1}^m\dfrac{\De\ol{\X}_{2i}(\De\ol{\X}_{2i})^T}{|\De\ol{\X}_{2i}|^2}
\]
have the same LSD. This will follow if we can show that
\begin{eqnarray}
\max_{i=1,\ldots,m} \left|\dfrac{|\De\ol{\Y}_{2i}|^2}{|\De\ol{\X}_{2i}|^2}-1\right| \to 0 ~~~~~ {\rm almost ~ surely},
\label{prop1}
\end{eqnarray}
and
\begin{eqnarray}
y_m\sum_{i=1}^m\dfrac{\De\ol{\Y}_{2i}(\De\ol{\Y}_{2i})^T}{|\De\ol{\X}_{2i}|^2}~~{\rm and ~~} \widetilde{\tpSi}~~~{\rm have~ the ~ same ~ LSD}.
\label{prop2}
\end{eqnarray}

We start with \eqref{prop1}. Since
\begin{eqnarray*}
\left|\dfrac{|\De\ol{\Y}_{2i}|^2}{|\De\ol{\X}_{2i}|^2}-1\right|
&=&\left|\dfrac{|\De\ol{\X}_{2i}|^2+|\De\ol{\pvep}_{2i}|^2+2(\De\ol{\X}_{2i})^T(\De\ol{\pvep}_{2i})}{|\De\ol{\X}_{2i}|^2}-1\right|\\
&\le&\(\dfrac{|\De\ol{\pvep}_{2i}|}{|\De\ol{\X}_{2i}|}\)^2+2\dfrac{|\De\ol{\pvep}_{2i}|}{|\De\ol{\X}_{2i}|},
\end{eqnarray*}
in order to prove \eqref{prop1}, it  suffices to show
\[
\max_{1\le i\le m} ~
\dfrac{|\De\ol{\pvep}_{2i}|}{|\De\ol{\X}_{2i}|} \to 0 ~~~~~~~~ {\rm almost ~surely}.
\]
Below we shall prove the following slightly stronger result:
\begin{eqnarray}\label{ep_x}
\max_{1\le i\le m,1\le j\le p} ~
\dfrac{\sqrt{p}~|\De\ol{\varepsilon}_{2i}^{j}|}{|\De\ol{\X}_{2i}|}~ \to~ 0 ~~~~~~ {\rm almost ~surely},
\end{eqnarray}
where for any vector ${\bf a}$,  $a^{j}$ denotes its $j$th entry.

Notice further that for \eqref{prop2},  using  Lemma \ref{lemma1} in  Appendix \ref{appendix:lemmas},  to prove \eqref{prop2}, it also suffices to show \eqref{ep_x}.

We now prove \eqref{ep_x}. We start with the denominator term $\De\ol{\X}_{2i}$.
We have $\De\ol{\X}_{2i}=\wt{\V}_i+\wt{\Z}_i$ for $\wt{\V}_i$ and $\wt{\Z}_i$ defined in \eqref{Vi} and \eqref{Zi} respectively.
Write $\wt{\Z}_i$ as $
\sqrt{w_i}
~ \pLa \Z_i$, where $w_i$ is defined in \eqref{wie} and
\begin{eqnarray*}
\Z_i=
\sum_{|j|<k}  \frac{1}{\sqrt{w_i}} \(1-\dfrac{|j|}{k}\)\int_{((2i-1)k+j-1)/n}^{((2i-1)k+j)/n} \ga_t\, d\W_t.
\end{eqnarray*}

 By Assumption \eqref{asm:leverage_2}, for all $j\notin \mathcal{I}_p$, $Z_i^{j}$ are \hbox{i.i.d.} $N(0,1)$. By using the same trick as the proof of (3.34) in \cite{ZL11}, we have
\begin{equation}\label{eq:mv_normal_norm}
\max_{1\le i\le m} ~
\left|\dfrac{1}{p}|\pLa \Z_i|^2 - 1\right| ~\to ~ 0 ~~~~~~{\rm almost~surely}.
\end{equation}
Note that
\begin{eqnarray*}
|\De\ol{\X}_{2i}|^2
~=~|\wt{\V}_i+\wt{\Z}_i|^2
~\geq~|\wt{\V}_i|^2+|\wt{\Z}_i|^2-2|\wt{\V}_i||\wt{\Z}_i|.
\end{eqnarray*}
Assumption \eqref{asm:gamma_bdd} implies that for  all $i$, there exist $\wt{C}_1
$ such that
\[
|w_i| ~\geq ~ \wt{C}_1\dfrac{k}{n}.
\]
Therefore by Assumption \eqref{asm:ym_conv}, there exists $C>0$ such that
\[
|\wt{\Z}_i|^2~=~
{|w_i|}
~|\pLa\Z_i|^2
~\geq~\dfrac{C}{p}|\pLa\Z_i|^2,
\]
which, together with \eqref{eq:mv_normal_norm}, implies that there exists $\delta_1>0$ such that for all $n$ large enough,
\[
\min_{1\le i\le m} ~ |\wt{\Z}_i|^2 \geq \delta_1.
\]
Moreover, by Assumption \eqref{asm:mu_bdd},
$|\wt{V}_i^{j}|\le C k/n$ for all $i,j$,  hence $\max_i |\wt{\V}_i|=O(\sqrt{p}\times k/n)$, which, by Assumption \eqref{asm:ym_conv}, is $O(\sqrt{1/m})=o(1)$.   Therefore, there exists a constant $\delta>0$ such that, almost surely, for all~$n$ large enough,
\begin{equation}\label{eq:delta_X_bdd}
\min_{1\le i\le m} ~ |\De\ol{\X}_{2i}|^2 ~\geq~ \delta.
\end{equation}

%
%

It remains to prove that
\begin{eqnarray}\label{lim_eps}
\max_{1\le i\le m,1\le j\le p} ~
\sqrt{p}~ | \De\ol{\varepsilon}_{2i}^{j}|~\to~ 0,  ~~~ {\rm almost ~ surely.}
\end{eqnarray}
Observe that if we can show  there exists $C>0$ such that
\begin{eqnarray}\label{dev_lim_eps}
\max_{1\le i\le m,1\le j\le p} E|\De\ol{\vep}_{2i}^j|^{2\ell} \ \le \ Ck^{-\ell},
\end{eqnarray}
where $\ell $ is the integer satisfying $\ell > (3-2\al)/(2\al-1)$ as in  Assumption \eqref{asm:ym_conv}, then for any $\varepsilon>0$, by Markov's inequality, we have
\begin{eqnarray*}
P\(\max_{1\le i\le m,1\le j\le p} \sqrt{p}~|\De\ol{\varepsilon}_{2i}^{j}|\geq\varepsilon\)
&\le& \sum_{1\le i\le m,1\le j\le p} \dfrac{p^\ell~ E|\De\ol{\varepsilon}_{2i}^{j}|^{2\ell}}{\varepsilon^{2\ell}}\\
&\le& \frac{C mp\cdot p^{\ell}}{k^\ell\varepsilon^{2\ell}} =
O\(\frac{1}{n^{(2+2\ell)\alpha - 2 - \ell}}\),
\end{eqnarray*}
where the last equation is due to Assumption \eqref{asm:ym_conv}. Since $\ell > (3-2\al)/(2\al-1)$, we have
$(2+2\ell)\alpha - 2 - \ell > 1$, hence by the Borel-Cantelli Lemma, \eqref{lim_eps} holds.

We now show \eqref{dev_lim_eps}, which is a Marcinkiewicz-Zygmund type inequality.
We will use Theorem 1 in \cite{DL99} to prove \eqref{dev_lim_eps}.
For that we need to verify $C_{r,2\ell}=O(r^{-\ell})$, where
\[
C_{r,2\ell}:= \max_{j=1,\ldots,p}\ \max_{1<M<2\ell} \sup_{(i_1,\cdots,i_{2\ell})\in\Theta_{r,M,2\ell}} \left| \cov( \vep_{i_1}^j \cdots \vep_{i_M}^j, \ \vep_{i_{M+1}}^j  \cdots \vep_{i_{2\ell}}^j ) \right|,
\]
where
$\Theta_{r,M,2\ell}=\{(i_1,\cdots,i_{2\ell}): i_1\le\ldots \le i_M < i_M+r \le i_{M+1}\le \ldots \le i_{2\ell}\}.$

We now verify that $C_{r,2\ell}=O(r^{-\ell})$.
For any $j$ and for any $(i_1,\cdots,i_{2\ell})\in\Theta_{r,M,2\ell}$, using the definition of $\rho$-mixing coefficients we have
\begin{eqnarray*}
&&\left| \cov( \vep_{i_1}^j  \cdots  \vep_{i_M}^j,\  \vep_{i_{M+1}}^j  \cdots \vep_{i_{2\ell}}^j ) \right| \\
&\le & \rho^j(r) \cdot \sqrt{ \var(\vep_{i_1}^j  \cdots  \vep_{i_M}^j) \cdot \var(\vep_{i_{M+1}}^j  \cdots \vep_{i_{2\ell}}^j) } \\
&\le & \rho^j(r) \cdot \sqrt{E((\vep_{i_1}^j)^2 \cdots (\vep_{i_M}^j)^2) \cdot E((\vep_{i_{M+1}}^j)^2\cdots (\vep_{i_{2\ell}}^j)^2 )}.
\end{eqnarray*}
By H\"{o}lder's inequality, we have
\[
E((\vep_{i_1}^j)^2 \cdots (\vep_{i_M}^j)^2) \ \le \ \(E(\vep_{i_1}^j)^{2M} \)^{1/M} \cdots \(E(\vep_{i_M}^j)^{2M} \)^{1/M}
\]
and similarly for $E((\vep_{i_{M+1}}^j)^2 \cdots (\vep_{i_{2\ell}}^j)^2 )$.
Since $(\vep_i^j)$'s have bounded $4\ell$th moments according to Assumption \eqref{asm:eps_general} and $\max_{j=1,\cdots,p} r^\ell  \rho^j(r)=O(1)$, we get $C_{r,2\ell}=O(r^{-\ell})$.

Finally, by using a similar argument as in the last part of the proof of Proposition 8 in \cite{ZL11} (see pp.3142--3143), we have that $\widetilde{\tpSi}$
has the same LSD as
\[
\widetilde{\S}:= \ \dfrac{1}{m}\sum_{i=1}^m \pLa \Z_i \Z_i^T \pLa^T,
\]
where $\Z_i$ consists of independent standard normals.
It is well known that the LSD of $\widetilde{\S}$ is determined by \eqref{eqn:B_n}, hence by the previous arguments, so does that of $\widetilde{\pSi}$.
\end{proof}

Now we prove Lemma \ref{pthm3_a}.
\begin{proof}[Proof of Lemma \ref{pthm3_a}]
We have 
\[
 \sum_{i=1}^m |\De\ol{\Y}_{2i}|^2
=\sum_{i=1}^m |\De\ol{\X}_{2i}|^2+ 2 \sum_{i=1}^m \De\ol{\X}_{2i}^T  \De\ol{\pvep}_{2i} + \sum_{i=1}^m|\De\ol{\pvep}_{2i}|^2.
\]
The convergence \eqref{lim_eps} and that $p/m\to y$ imply that $\sum_{i=1}^m|\De\ol{\pvep}_{2i}|^2/p\to 0$ almost surely.
To prove the lemma, it then suffices to show that 
\begin{equation}\label{subpav_conv}
  3\frac{\sum_{i=1}^m |\De\ol{\X}_{2i}|^2}{p}\to \zeta\q\mbox{almost surely},
\end{equation}
and 
\begin{equation}\label{subpav_error_negligible}
  \frac{\sum_{i=1}^m \De\ol{\X}_{2i}^T  \De\ol{\pvep}_{2i}}{p}\to 0\q\mbox{almost surely}.
\end{equation}

We start with \eqref{subpav_conv}. Write $\De\ol{\X}_{2i}=\wt{\V}_i+\wt{\Z}_i$  as in the proof of Proposition~\ref{pthm3}. The convergence \eqref{eq:mv_normal_norm} implies that
\[
\frac{\sum_{i=1}^m |\wt{\Z}_i|^2}{p} = \sum_{i=1}^m w_i + \mbox{error}, 
\]
where the error term converges to 0 almost surely. By Riemann integration and Assumption \eqref{asm:gamma_conv} it is easy to show that  $\sum_{i=1}^m w_i\to \zeta/3$, so we get
\[
3\frac{\sum_{i=1}^m |\wt{\Z}_i|^2}{p} \to \zeta \q\mbox{almost surely}.
\]
Furthermore by using the bound that $\max_i |\wt{\V}_i|=O(\sqrt{p}\times k/n)$ one can easily show that 
\[
\frac{2\sum_{i=1}^m |\wt{\V}_i||\wt{\Z}_i| +  \sum_{i=1}^m |\wt{\V}_i|^2}{p} \to   0\q\mbox{almost surely}.
\]
We therefore get \eqref{subpav_conv}.

Finally, \eqref{subpav_error_negligible} follows from \eqref{lim_eps} and \eqref{subpav_conv}.
\end{proof}


\renewcommand{\baselinestretch}{1.2}
\section{Two lemmas}\label{appendix:lemmas}
\setcounter{equation}{0}
\renewcommand{\theequation}{\thesection.\arabic{equation}}

\begin{lem}(Lemma 2.7 in \cite{BS98} ).
Let ${\bf X}=(X_1,\ldots,X_n)^T$ be a vector where the $X_i$'s are centered \hbox{i.i.d.} random variables with unit variance. Let ${\bf A}$ be an $n\times n$ deterministic complex matrix. Then, for any $p\geq 2$,
\[
E\left|{\bf X}^T{\bf A}{\bf X}-\tr{\bf A}\right|^p\le C_p
\left(
\left(E|X_1|^4\tr{\bf A}{\bf A}^*\right)^{p/2}+E|X_1|^{2p}\tr({\bf A}{\bf A}^*)^{p/2}
\right).
\]
\label{xtrx}
\end{lem}

\begin{lem}(Lemma 1 in \cite{ZL11}).
Suppose that for each $p$, ${\bf v}_l=(v_l^{1},\ldots,v_l^{p})^T$ and ${\bf w}_l=(w_l^{1},\ldots,w_l^{p})^T$, $l=1,\ldots,m$, are all $p$-dimensional vectors. Define
\[
\wt{{\bf S}}_m=\sum_{l=1}^m({\bf v}_l+{\bf w}_l)({\bf v}_l+{\bf w}_l)^T
~~ and ~~
{\bf S}_m=\sum_{l=1}^m{\bf w}_l({\bf w}_l)^T.
\]
If the following conditions are satisfied:
\begin{itemize}
  \item $m=m(p)$ with $\lim_{p\to\infty} p/m=y>0$;
  \item there exists a sequence $\varepsilon_p=o(1/\sqrt{p})$ such that for all $p$ and all $l$, all the entries of ${\bf v}_l$ are bounded by $\varepsilon_p$ in absolute value;
  \item $\limsup_{p\to\infty} \tr({\bf S}_m)/p<\infty$ almost surely.
\end{itemize}
Then $L(F^{\wt{S}_m},F^{S_m})\to 0$ almost surely, where for any two probability distribution functions $F$ and $G$, $L(F,G)$ denotes the Levy distance between them.
\label{lemma1}
\end{lem}



\begin{thebibliography}{20}

\bibitem[\protect\citeauthoryear{Andersen and Bollerslev}{1998}]{AB98}
\begin{barticle}[author]
\bauthor{\bsnm{Andersen},~\binits{T.~G}},
 \AND
  \bauthor{\bsnm{Bollerslev},~\binits{T.}}
(\byear{1998}).
\btitle{Answering the Skeptics: Yes, Standard Volatility
Models Do Provide Accurate Forecasts}.
\bjournal{International Economic Review}
\bvolume{39}
\bpages{885--905}.
\end{barticle}
\endbibitem

\bibitem[\protect\citeauthoryear{Andersen et~al.}{2001}]{ABDL01}
\begin{barticle}[author]
\bauthor{\bsnm{Andersen},~\binits{T.~G.}},
  \bauthor{\bsnm{Bollerslev},~\binits{T.}},
   \bauthor{\bsnm{Diebold},~\binits{F.~X.}}
  \AND
  \bauthor{\bsnm{Labys},~\binits{P.}}
(\byear{2000}).
\btitle{The Distribution of Realized Exchange Rate Volatility}.
\bjournal{Journal of the American Statistical Association}
\bvolume{96}
\bpages{42--55}.
\end{barticle}
\endbibitem


\bibitem[\protect\citeauthoryear{A{\"\i}t-Sahalia, Fan and
  Li}{2013}]{AFL13}
\begin{barticle}[author]
\bauthor{\bsnm{A{\"\i}t-Sahalia},~\bfnm{Yacine}\binits{Y.}},
  \bauthor{\bsnm{Fan},~\bfnm{Jianqing}\binits{J.}} \AND
  \bauthor{\bsnm{Li},~\bfnm{Yingying}\binits{Y.}}
(\byear{2010}).
\btitle{The leverage effect puzzle: Disentangling sources of bias
at high frequency}.
\bjournal{Journal of Financial Economics}
\bvolume{109}
\bpages{224--249}.
\end{barticle}
\endbibitem

\bibitem[A\"{\i}t-Sahalia \etal(2010)]{AFX10}
 A\"{\i}t-Sahalia, Y.,   Fan, J.   Xiu, D.  (2010). High Frequency
Covariance Estimates with Noisy and Asynchronous Financial Data,  {\em Journal of the American Statistical Association},  {\bf  105}, 1504-1517.



\bibitem[\protect\citeauthoryear{Bai, Chen and Yao}{2010}]{BCY10}
\begin{barticle}[author]
\bauthor{\bsnm{Bai},~\bfnm{Zhidong}\binits{Z.}},
  \bauthor{\bsnm{Chen},~\bfnm{Jiaqi}\binits{J.}} \AND
  \bauthor{\bsnm{Yao},~\bfnm{Jianfeng}\binits{J.}}
(\byear{2010}).
\btitle{On estimation of the population spectral distribution from a
  high-dimensional sample covariance matrix}.
\bjournal{Aust. N. Z. J. Stat.}
\bvolume{52}
\bpages{423--437}.
\bdoi{10.1111/j.1467-842X.2010.00590.x}.
\bmrnumber{2791528 (2012c:62157)}
\end{barticle}
\endbibitem


\bibitem[\protect\citeauthoryear{Bai and Silverstein}{1998}]{BS98}
\begin{barticle}[author]
\bauthor{\bsnm{Bai},~\bfnm{Z.~D.}\binits{Z.~D.}} \AND
  \bauthor{\bsnm{Silverstein},~\bfnm{Jack~W.}\binits{J.~W.}}
(\byear{1998}).
\btitle{No eigenvalues outside the support of the limiting spectral
  distribution of large-dimensional sample covariance matrices}.
\bjournal{Ann. Probab.}
\bvolume{26}
\bpages{316--345}.
\bdoi{10.1214/aop/1022855421}.
\bmrnumber{1617051}
\end{barticle}
\endbibitem

\bibitem[\protect\citeauthoryear{Bai and Silverstein}{2012}]{Bai2012_noise}
\begin{barticle}[author]
\bauthor{\bsnm{Bai},~\bfnm{Zhidong}\binits{Z.}} \AND
  \bauthor{\bsnm{Silverstein},~\bfnm{Jack~W.}\binits{J.~W.}}
(\byear{2012}).
\btitle{No eigenvalues outside the support of the limiting spectral
  distribution of information-plus-noise type matrices}.
\bjournal{Random Matrices Theory Appl.}
\bvolume{1}
\bpages{1150004, 44}.
\bdoi{10.1142/S2010326311500043}.
\bmrnumber{2930382}
\end{barticle}
\endbibitem



\bibitem[Barndorff-Nielsen \etal(2008)]{BHLS08}
Barndorff-Nielsen, O. E.,   Hansen, P. R., Lunde A.
    and Shephard N. (2008).
Designing Realized Kernels to Measure Ex-post Variation of Equity
    Prices in the Presence of Noise,
 \textit{Econometrica},
  {\bf 76},
 1481-1536.


\bibitem[\protect\citeauthoryear{Barndorff-Nielsen and  Shephard}{2002}]{BNS02}
\begin{barticle}[author]
\bauthor{\bsnm{Barndorff-Nielsen},~\bfnm{Ole~E.}\binits{O.~E.}},
 \AND
  \bauthor{\bsnm{Shephard},~\bfnm{Neil}\binits{N.}}
(\byear{2002}).
\btitle{Econometric Analysis of Realized Volatility and its use in
              Estimating Stochastic Volatility Models}.
\bjournal{Journal of the Royal Statistical Society. Series B.
              Statistical Methodology}
\bvolume{64}
\bpages{253--280}.
\end{barticle}
\endbibitem



\bibitem[\protect\citeauthoryear{Barndorff-Nielsen et~al.}{2011}]{BHLS11}
\begin{barticle}[author]
\bauthor{\bsnm{Barndorff-Nielsen},~\bfnm{Ole~E.}\binits{O.~E.}},
  \bauthor{\bsnm{Hansen},~\bfnm{Peter~Reinhard}\binits{P.~R.}},
  \bauthor{\bsnm{Lunde},~\bfnm{Asger}\binits{A.}} \AND
  \bauthor{\bsnm{Shephard},~\bfnm{Neil}\binits{N.}}
(\byear{2011}).
\btitle{Multivariate Realised Kernels: Consistent Positive Semi-Definite
  Estimators of the Covariation of Equity Prices with Noise and Non-Synchronous
  Trading}.
\bjournal{Journal of Econometrics}
\bvolume{162}
\bpages{149-169}.
\end{barticle}
\endbibitem




\bibitem[\protect\citeauthoryear{Christensen, Kinnebrock and Podolskij}{2010}]{CKP10}
\begin{barticle}[author]
\bauthor{\bsnm{Christensen},~\bfnm{Kim}\binits{K.}},
\bauthor{\bsnm{Kinnebrock},~\bfnm{Silja}\binits{S.}}
 \AND
  \bauthor{\bsnm{Podolskij},~\bfnm{Mark}\binits{M.}}
(\byear{2010}).
\btitle{Pre-averaging estimators of the ex-post covariance matrix in
              noisy diffusion models with non-synchronous data}.
\bjournal{J. Econometrics}
\bvolume{159}
\bpages{116--133}.
\bdoi{10.1016/j.jeconom.2010.05.001}.
\end{barticle}
\endbibitem


\bibitem[\protect\citeauthoryear{Doukhan and Louhichi}{1999}]{DL99}
\begin{barticle}[author]
\bauthor{\bsnm{Doukhan},~\bfnm{Paul}\binits{P.}} \AND
  \bauthor{\bsnm{Louhichi},~\bfnm{Sana}\binits{S.}}
(\byear{1999}).
\btitle{A new weak dependence condition and applications to moment inequalities}.
\bjournal{Stochastic Processes and their Applications}
\bvolume{84}
\bpages{313--342}
\bdoi{10.1016/S0304-4149(99)00055-1}.
\bmrnumber{1719345}
\end{barticle}
\endbibitem


\bibitem[\protect\citeauthoryear{Dozier and Silverstein}{2007a}]{DS2007a}
\begin{barticle}[author]
\bauthor{\bsnm{Dozier},~\bfnm{R.~Brent}\binits{R.~B.}} \AND
  \bauthor{\bsnm{Silverstein},~\bfnm{Jack~W.}\binits{J.~W.}}
(\byear{2007}a).
\btitle{On the empirical distribution of eigenvalues of large dimensional
  information-plus-noise-type matrices}.
\bjournal{J. Multivariate Anal.}
\bvolume{98}
\bpages{678--694}.
\bdoi{10.1016/j.jmva.2006.09.006}.
\bmrnumber{2322123}
\end{barticle}
\endbibitem

\bibitem[\protect\citeauthoryear{Dozier and Silverstein}{2007b}]{DS2007b}
\begin{barticle}[author]
\bauthor{\bsnm{Dozier},~\bfnm{R.~Brent}\binits{R.~B.}} \AND
  \bauthor{\bsnm{Silverstein},~\bfnm{Jack~W.}\binits{J.~W.}}
(\byear{2007}b).
\btitle{Analysis of the limiting spectral distribution of large dimensional
  information-plus-noise type matrices}.
\bjournal{J. Multivariate Anal.}
\bvolume{98}
\bpages{1099--1122}.
\bdoi{10.1016/j.jmva.2006.12.005}.
\bmrnumber{2326242}
\end{barticle}
\endbibitem



\bibitem[\protect\citeauthoryear{El~Karoui}{2008}]{ElKaroui08}
\begin{barticle}[author]
\bauthor{\bsnm{El~Karoui},~\bfnm{Noureddine}\binits{N.}}
(\byear{2008}).
\btitle{Spectrum estimation for large dimensional covariance matrices using
  random matrix theory}.
\bjournal{Ann. Statist.}
\bvolume{36}
\bpages{2757--2790}.
\bdoi{10.1214/07-AOS581}.
\bmrnumber{2485012}
\end{barticle}
\endbibitem


\bibitem[\protect\citeauthoryear{El~Karoui}{2009}]{ElKaroui09}
\begin{barticle}[author]
\bauthor{\bsnm{El~Karoui},~\bfnm{Noureddine}\binits{N.}}
(\byear{2009}).
\btitle{Concentration of measure and spectra of random matrices:
              applications to correlation matrices, elliptical distributions
              and beyond}.
\bjournal{Ann. Appl. Probab.}
\bvolume{19}
\bpages{2362--2405}.
\bdoi{10.1214/08-AAP548}.
\bmrnumber{2588248 }
\end{barticle}
\endbibitem

\bibitem[\protect\citeauthoryear{El~Karoui}{2010}]{ElKaroui10}
\begin{barticle}[author]
\bauthor{\bsnm{El~Karoui},~\bfnm{Noureddine}\binits{N.}}
(\byear{2010}).
\btitle{High-dimensionality effects in the {M}arkowitz problem and
              other quadratic programs with linear constraints: risk
              underestimation}.
\bjournal{Ann. Statist.}
\bvolume{38}
\bpages{3487--3566}.
\bdoi{10.1214/10-AOS795}.
\bmrnumber{2766860 }
\end{barticle}
\endbibitem

\bibitem[\protect\citeauthoryear{El~Karoui}{2013}]{ElKaroui13}
\begin{barticle}[author]
\bauthor{\bsnm{El~Karoui},~\bfnm{Noureddine}\binits{N.}}
(\byear{2013}).
\btitle{On the realized risk of high-dimensional {M}arkowitz
              portfolios}.
\bjournal{SIAM J. Financial Math.}
\bvolume{4}
\bpages{737--783}.
\bdoi{10.1137/090774926}.
\bmrnumber{3118251 }
\end{barticle}
\endbibitem




\bibitem[\protect\citeauthoryear{Hansen and Lunde}{2006}]{hansenlunde06}
  Hansen, P. R. and  Lunde, A.(2006), {Realized Variance and Market
  Microstructure Noise,}
 \textit{Journal of Business and Economic Statistics},
   {\bf 24}, 127--161.


\bibitem[\protect\citeauthoryear{Jacod et~al.}{2009}]{JLMPV09}
\begin{barticle}[author]
\bauthor{\bsnm{Jacod},~\bfnm{Jean}\binits{J.}},
  \bauthor{\bsnm{Li},~\bfnm{Yingying}\binits{Y.}},
  \bauthor{\bsnm{Mykland},~\bfnm{Per~A.}\binits{P.~A.}},
  \bauthor{\bsnm{Podolskij},~\bfnm{Mark}\binits{M.}} \AND
  \bauthor{\bsnm{Vetter},~\bfnm{Mathias}\binits{M.}}
(\byear{2009}).
\btitle{Microstructure noise in the continuous case: the pre-averaging
  approach}.
\bjournal{Stochastic Process. Appl.}
\bvolume{119}
\bpages{2249--2276}.
\bdoi{10.1016/j.spa.2008.11.004}.
\bmrnumber{2531091}
\end{barticle}
\endbibitem

\bibitem[\protect\citeauthoryear{Jacod and Protter}{1998}]{jacodprotter98}
\begin{barticle}[author]
\bauthor{\bsnm{Jacod},~\binits{J.}} \AND
  \bauthor{\bsnm{Protter},~\binits{P.}}
(\byear{1998}).
\btitle{Asymptotic Error Distributions for the Euler Method for Stochastic
  Differential Equations}.
\bjournal{Annals of Probability}
\bvolume{26}
\bpages{267--307}.
\end{barticle}
\endbibitem

\bibitem[\protect\citeauthoryear{Jocod et~al.}{2014}]{JLZ_noise_V2} Jacod J., Y. Li, and X. Zheng (2014). Statistical Properties of Microstructure Noise: II. Available at SSRN: \textit{http://ssrn.com/abstract=2212119}.


\bibitem[\protect\citeauthoryear{Liu et~al.}{2015}]{LPS15}
\begin{barticle}[author]
\bauthor{\bsnm{Liu},~\bfnm{Lily Y.}\binits{L.Y.}},
  \bauthor{\bsnm{Patton},~\bfnm{Andrew J.}\binits{A.J.}} \AND
  \bauthor{\bsnm{Sheppard},~\bfnm{Kevin}\binits{K.}}
(\byear{2015}).
\btitle{Does Anything Beat 5-Minute RV?
A Comparison of Realized Measures Across Multiple Asset Classes}.
\bjournal{Journal of Econometrics}.
\bvolume{187}
\bpages{293--311}.
\bdoi{10.1016/j.jeconom.2015.02.008}.
\bmrnumber{3347308}
\end{barticle}
\endbibitem



\bibitem[\protect\citeauthoryear{McNeil, Frey and Embrechts}{2005}]{MFE05}
\begin{barticle}[author]
\bauthor{\bsnm{McNeil},~\bfnm{Alexander J.}\binits{A.J.}},
  \bauthor{\bsnm{Frey},~\bfnm{Rudiger}\binits{R.}} \AND
  \bauthor{\bsnm{Embrechts},~\bfnm{Paul}\binits{P.}}
(\byear{2005}).
\btitle{Quantitative Risk Management: Concepts, Techniques, and Tools.}.
\bjournal{Princeton University Press}.
\end{barticle}
\endbibitem


\bibitem[\protect\citeauthoryear{Mestre}{2008}]{Mestre08}
\begin{barticle}[author]
\bauthor{\bsnm{Mestre},~\bfnm{Xavier}\binits{X.}}
(\byear{2008}).
\btitle{Improved estimation of eigenvalues and eigenvectors of covariance
  matrices using their sample estimates}.
\bjournal{IEEE Trans. Inform. Theory}
\bvolume{54}
\bpages{5113--5129}.
\bdoi{10.1109/TIT.2008.929938}.
\bmrnumber{2589886 (2010h:62174)}
\end{barticle}
\endbibitem


\bibitem[\protect\citeauthoryear{Mykland and Zhang}{2006}]{myklandzhang06}
\begin{barticle}[author]
\bauthor{\bsnm{Mykland},~\binits{P.~A.}} \AND
  \bauthor{\bsnm{Zhang},~\binits{L.}}
(\byear{2006}).
\btitle{ANOVA for Diffusions and Ito Processes}.
\bjournal{Annals of Statistics}
\bvolume{34}
\bpages{1931-1963}.
\end{barticle}
\endbibitem

\bibitem[\protect\citeauthoryear{Podolskij and Vetter}{2009}]{PV09}
\begin{barticle}[author]
\bauthor{\bsnm{Podolskij},~\bfnm{Mark}\binits{M.}} \AND
  \bauthor{\bsnm{Vetter},~\bfnm{Mathias}\binits{M.}}
(\byear{2009}).
\btitle{Estimation of volatility functionals in the simultaneous presence of
  microstructure noise and jumps}.
\bjournal{Bernoulli}
\bvolume{15}
\bpages{634--658}.
\bdoi{10.3150/08-BEJ167}.
\bmrnumber{2555193 (2011b:62081)}
\end{barticle}
\endbibitem



\bibitem[\protect\citeauthoryear{Silverstein and Bai}{1995}]{SB95}
\begin{barticle}[author]
\bauthor{\bsnm{Silverstein},~\bfnm{Jack~W.}\binits{J.~W.}} \AND
  \bauthor{\bsnm{Bai},~\bfnm{Z.~D.}\binits{Z.~D.}}
(\byear{1995}).
\btitle{On the empirical distribution of eigenvalues of a class of
  large-dimensional random matrices}.
\bjournal{J. Multivariate Anal.}
\bvolume{54}
\bpages{175--192}.
\bdoi{10.1006/jmv\ref{xtrx}995.1051}.
\bmrnumber{1345534}
\end{barticle}
\endbibitem


\bibitem[\protect\citeauthoryear{Ukabata and Oya}{2009}]{UO09}
Ukabata, M. and K. Oya (2009).
Estimation and Testing for Dependence in Market Microstructure Noise.
\textit{J. Financial Econometrics}, \textbf{7}, 106-151.


\bibitem[\protect\citeauthoryear{Wang and Mykland}{2014}]{WM14}
\begin{barticle}[author]
\bauthor{\bsnm{Wang},~\bfnm{Christina D.}\binits{C.D.}} \AND
  \bauthor{\bsnm{Mykland},~\bfnm{Per A.}\binits{P.A.}}
(\byear{2014}).
\btitle{The Estimation of Leverage Effect With High-Frequency Data}.
\bjournal{J. Amer. Statist. Assoc.}
\bvolume{109}
\bpages{197--215}.
\bdoi{10.1080/01621459.2013.864189}.
\bmrnumber{3180557}
\end{barticle}
\endbibitem






\bibitem[\protect\citeauthoryear{Xiu}{2010}]{Xiu10}
\begin{barticle}[author]
\bauthor{\bsnm{Xiu},~\bfnm{Dacheng}\binits{D.}}
(\byear{2010}).
\btitle{Quasi-maximum likelihood estimation of volatility with high frequency
  data}.
\bjournal{J. Econometrics}
\bvolume{159}
\bpages{235--250}.
\bdoi{10.1016/j.jeconom.2010.07.002}.
\bmrnumber{2720855 (2011f:62036)}
\end{barticle}
\endbibitem



\bibitem[Zhang \etal(2005)]{ZMA05}
    Zhang, L., Mykland, P. A. and A\"{\i}t-Sahalia, Y. (2005).
    A Tale of Two Time Scales: Determining Integrated Volatility with
        Noisy High-Frequency Data," \textit{Journal of the American Statistical
        Association}, {\bf 100}, 1394-1411.


\bibitem[\protect\citeauthoryear{Zhang}{2011}]{Zhang11}
\begin{barticle}[author]
\bauthor{\bsnm{Zhang},~\bfnm{Lan}\binits{L.}}
(\byear{2011}).
\btitle{Estimating covariation: Epps effect, microstructure noise}.
\bjournal{Journal of Econometrics}
\bvolume{160}
\bpages{33--47}.
\bdoi{10.1016/j.jeconom.2010.03.012}.
\bmrnumber{2745865 (2012b:62345)}
\end{barticle}
\endbibitem



\bibitem[\protect\citeauthoryear{Zhang, Mykland and
  A{\"{\i}}t-Sahalia}{2005}]{ZMA05}
\begin{barticle}[author]
\bauthor{\bsnm{Zhang},~\bfnm{Lan}\binits{L.}},
  \bauthor{\bsnm{Mykland},~\bfnm{Per~A.}\binits{P.~A.}} \AND
  \bauthor{\bsnm{A{\"{\i}}t-Sahalia},~\bfnm{Yacine}\binits{Y.}}
(\byear{2005}).
\btitle{A tale of two time scales: determining integrated volatility with noisy
  high-frequency data}.
\bjournal{J. Amer. Statist. Assoc.}
\bvolume{100}
\bpages{1394--1411}.
\bdoi{10.1198/016214505000000169}.
\bmrnumber{2236450 (2007h:62079)}
\end{barticle}
\endbibitem

\bibitem[\protect\citeauthoryear{Zheng and Li}{2011}]{ZL11}
\begin{barticle}[author]
\bauthor{\bsnm{Zheng},~\bfnm{Xinghua}\binits{X.}} \AND
  \bauthor{\bsnm{Li},~\bfnm{Yingying}\binits{Y.}}
(\byear{2011}).
\btitle{On the estimation of integrated covariance matrices of high dimensional
  diffusion processes}.
\bjournal{Ann. Statist.}
\bvolume{39}
\bpages{3121--3151}.
\bdoi{10.1214/11-AOS939}.
\bmrnumber{3012403}
\end{barticle}
\endbibitem





\end{thebibliography}


\end{document}